\documentclass[12pt,a4paper]{amsart}
\usepackage{amsmath}
\usepackage{amsfonts}
\usepackage{amssymb}
\usepackage{graphicx,a4wide}
\usepackage[font=footnotesize]{subfig}

\newcommand{\tfs}{time-frequency shift}
\def\lrd{L^2(\Rdst)}
\def\rd{\Rdst}
\def\rdd{\Rtdst}
\newcommand{\tfa}{time-frequency analysis}
\def\cG{\mathcal{G}}
\newcommand{\modsp}{modulation space}
\newcommand{\vs}{\vspace{3 mm}}
\newcommand{\fif}{if and only if}
\newcommand{\ft}{Fourier transform}
\newcommand{\tf}{time-frequency}
\def\cS{\mathcal{S}}

\newcommand{\bN}{\Nst}
\newcommand{\bZ}{\Zst}
\def\intrdd{\intrd{\rdd}}
\def\zdd{{\bZ^{2d}}}
\def\zd{\bZ^d}
\def\inv{^{-1}}
\def\cW{\mathcal{W}}
\def\intrdd{\int_{\rdd}}

\newcommand{\map}{\tau}
\newcommand{\mapn}[1]{\map_n(#1)}

\newcommand{\modzero}{M^0}

\newcommand{\regone}{(L1)}
\newcommand{\regtwo}{(L2)}
\newcommand{\covnum}{\eta}
\newcommand{\weakconv}{\xrightarrow{w}}
\newcommand{\lipconv}{\xrightarrow{Lip}}
\newcommand{\Rdst}{{\mathbb{R}^d}}
\newcommand{\Rst}{{\mathbb{R}}}
\newcommand{\Zst}{{\mathbb{Z}}}
\newcommand{\Nst}{{\mathbb{N}}}
\newcommand{\Rtdst}{{\mathbb{R}^{2d}}}
\newcommand{\Zdst}{{\mathbb{Z}^d}}
\newcommand{\Ztdst}{{\mathbb{Z}^{2d}}}
\newcommand{\LtRd}{{L^2(\Rdst)}}
\newcommand{\LtRtd}{{L^2(\Rtdst)}}
\newcommand{\norm}[1]{\lVert#1\rVert}
\newcommand{\bC}{{\mathbb{C}}}
\newcommand{\set}[2]{\big\{ \, #1 \, :  \, #2 \, \big\}}
\def\Shah{{\makebox[2.3ex][s]{$\sqcup$\hspace{-0.15em}\hfill $\sqcup$}\, \, }}
\def\supp{\operatorname{supp}}
\newcommand{\Bignorm}[1]{\Bigl\|#1\Bigr\|}   
\newcommand{\bignorm}[1]{\bigl\lVert#1\bigr\rVert}

\newcommand{\ip}[2]{\ensuremath{\left<#1,#2\right>}}
\newcommand{\sett}[1]{\ensuremath{\left \{ #1 \right \}}}
\newcommand{\abs}[1]{\ensuremath{\left| #1 \right| }}

\newcommand{\gab}{\mathcal{G}(g,\Lambda)}
\newcommand{\ltrd}{{L^2(\Rdst)}}
\newcommand{\group}{G}

\newcommand{\stft}{V}
\newcommand{\env}{\Theta}
\newcommand{\menv}{\triangle}
\newcommand{\rel}{\mathop{\mathrm{rel}}}
\newcommand{\gap}{\rho}
\newcommand{\sep}{\mathop{\mathrm{sep}}}
\newcommand{\twistshift}{\kappa}

\newcommand{\wmes}{{W(\mathcal{M}, L^\infty)}}
\newcommand{\wtest}{{W(C_0,L^1)}}
\newcommand{\Lambdastar}{\Gamma}
\newcommand{\Step}[1]{\smallskip {\bf #1}}

\newtheorem{lemma}{Lemma}[section]
\newtheorem{theo}[lemma]{Theorem}
\newtheorem*{theo*}{Theorem}
\newtheorem{coro}[lemma]{Corollary}
\newtheorem{prop}[lemma]{Proposition}

\newtheorem{rem}[lemma]{Remark}
\newtheorem{definition}[lemma]{Definition}
\newtheorem{example}[lemma]{Example}

\title{Deformation of Gabor systems}
\author{Karlheinz Gr\"ochenig}
\address{Faculty of Mathematics \\
University of Vienna \\
Oskar-Morgenstern-Platz 1 \\
A-1090 Vienna, Austria}
\email{karlheinz.groechenig@univie.ac.at}
\author{Joaquim Ortega-Cerd\`a}
\address{Dept.\ Matem\`atica Aplicada i An\`alisi,
Universitat  de Barcelona, Gran Via 585, 08007 Bar\-ce\-lo\-na, Spain}
\email{jortega@ub.edu}
\author{Jos\'e Luis Romero}
\address{Faculty of Mathematics \\
University of Vienna \\
Oskar-Morgenstern-Platz 1 \\
A-1090 Vienna, Austria}
\email{jose.luis.romero@univie.ac.at}
\thanks{K.\ G. was  supported in part by the project P26273 - N25  and
  the National Research Network S106 SISE of the
Austrian Science Fund (FWF). J.~O.-C.~ is supported by the Generalitat de
Catalunya (grant 2014SGR289) and the Spanish Ministerio de Econom\'ia y
Competividad (project MTM2011-27932-C02-01). J. L. R. gratefully acknowledges support from the project M1586-N25 of the
Austrian Science Fund (FWF) and from an individual Marie Curie fellowship, within the 7th. European Community Framework
program, under grant PIIF-GA-2012-327063.}
\begin{document}
\begin{abstract}
We introduce a new notion for the  deformation of Gabor systems. Such
deformations are in general nonlinear and, in particular, include   the standard
jitter error and linear deformations of phase space. With   this
new  notion we  prove a strong deformation result for Gabor
 frames and Gabor Riesz sequences that  covers the
known perturbation and deformation results.  Our 
proof of the deformation theorem requires a 
new characterization of Gabor frames and Gabor Riesz sequences. It is  in the style of 
Beurling's characterization of sets of sampling for bandlimited
functions and extends significantly the known  characterization of Gabor
frames ``without inequalities'' from lattices to non-uniform sets. 
\end{abstract}

\maketitle
\section{Introduction}

The question of robustness of a basis or frame is a  fundamental problem
in functional analysis and in many concrete applications. It has its
historical origin  in the work of Paley and Wiener~\cite{young80} who studied
the perturbation of Fourier  bases and was subsequently investigated
in many contexts in complex analysis and harmonic
analysis. Particularly fruitful was the study of the robustness of  structured function
systems, such as reproducing kernels,  sets of  sampling in a space of analytic
functions, wavelets, or Gabor systems. 
 In this paper we take a new look at the stability of  Gabor frames
 and Gabor Riesz sequences with respect to general deformations of
 phase space. 

To be  explicit, let us  denote the \tfs\ of a function
$g\in \lrd $ along   $z= (x,\xi )
\in \rd \times \rd \simeq \rdd $  by 
$$
\pi (z) g(t) = e^{2\pi i \xi \cdot t} g(t-x) \, .
$$
For a fixed non-zero function  $g \in \LtRd$, usually called a
``window function'',   and $\Lambda \subseteq \Rtdst$, a  Gabor system is
a structured function system of the form
\begin{align*}
\mathcal{G}(g,\Lambda)
= \sett{\pi(\lambda)g := e^{2\pi i \xi \cdot}g(\cdot-x): \lambda =
  (x,\xi) \in \Lambda}\, . 
\end{align*}
The index  set $\Lambda $ is a discrete subset of the phase space
$\rdd $ and $\lambda $ indicates the localization of a \tfs\ $\pi
(\lambda )g$ in phase space. 

The Gabor system  $\mathcal{G}(g,\Lambda)$
is called a \emph{frame} (a Gabor frame),  if
\begin{align*}
A \norm{f}_2^2 \leq \sum_{\lambda \in \Lambda} \abs{\ip{f}{\pi(\lambda)g}}^2
\leq B \norm{f}_2^2,
\qquad f \in \LtRd,
\end{align*}
for some constants $0<A\leq B < \infty$. In this case every function
$f \in \LtRd$ possesses  an expansion $f=\sum_{\lambda} c_\lambda \pi(\lambda)g$, 
for some coefficient sequence $c \in \ell^2(\Lambda)$ such that $\norm{f}_2 \asymp \norm{c}_2$.
The Gabor system $\mathcal{G}(g,\Lambda)$ is called a \emph{Riesz
  sequence} (or Riesz basis for its span), 
if $\norm{\sum_{\lambda} c_\lambda \pi(\lambda)g}_2 \asymp \norm{c}_2$ for all $c \in \ell^2(\Lambda)$.

For meaningful statements about Gabor frames it is usually assumed
that
\begin{align*}
\int_{\Rtdst} \abs{\ip{g}{\pi(z)g}} dz < \infty.
\end{align*}
This condition describes  the \modsp\ $M^1(\rd )$, also known as the Feichtinger algebra.
Every  Schwartz function satisfies this condition. 

In this paper    we study  the stability of the spanning
properties of $\mathcal{G}(g,\Lambda)$ with respect to  a set
$\Lambda \subseteq \rdd $. 
 If $\Lambda '$ is ``close enough'' to $\Lambda $, then we expect $\cG
 (g,\Lambda ')$ to possess the same spanning properties. In this
 context we distinguish perturbations and deformations.  Whereas a 
 perturbation is local and $\Lambda '$ is obtained by slightly moving
 every   $\lambda \in \Lambda $, a deformation is a global
 transformation of $\rdd $.  
The existing literature is rich in perturbation results, but not much
is known about deformations of Gabor frames. 

(a) \emph{Perturbation or jitter error:} The jitter describes  small pointwise
perturbations of $\Lambda $. For every Gabor frame $\gab $ with $g\in
M^1(\rd )$ there exists a maximal jitter $\epsilon >0$ with the
following property: if $\sup _{\lambda \in \Lambda } \inf _{\lambda '
    \in \Lambda '} |\lambda - \lambda '|<\epsilon$ and 
$\sup _{\lambda' \in \Lambda' } \inf _{\lambda
    \in \Lambda} |\lambda - \lambda '|
 < \epsilon $, then  $\cG (g, \Lambda ')$ is also a frame. 
See~\cite{fegr89,gr91} for a general result in coorbit theory, the recent
paper~\cite{FS06}, and Christensen's
book on frames~\cite{chr03} for more details and references. 

Conceptually the jitter error is easy to
understand, because the frame operator is continuous in the operator
norm with respect to the jitter error.  The proof techniques go back
to Paley and Wiener~\cite{young80} and  amount to norm
estimates for the frame operator.

(b) \emph{Linear deformations:} The fundamental deformation result is
due to Feichtinger and Kaiblinger~\cite{feka04}. Let  $g\in M^1(\rd )$, 
$\Lambda \subseteq \rdd $ be a lattice,  and assume that $\gab $ is a
frame for $\lrd $. Then there exists  $\epsilon >0$ with
the following property: if $A$ is a $2d\times 2d$-matrix with
$\|A-\mathrm{I}\| <\epsilon $ (in some given  matrix norm), then $\cG (g,
A\Lambda )$ is again a frame.  Only recently, this result was generalized to non-uniform
Gabor frames~\cite{asfeka13}. 
The proof for the case of a lattice~\cite{feka04}  was based on  the
duality  theory of Gabor frames, the proof for non-uniform Gabor
frames in~\cite{asfeka13} relies
on the stability under chirps of the Sj\"ostrand symbol class for
pseudodifferential operators, but  this  
technique does not seem to adapt to nonlinear deformations.
Compared to perturbations, (linear)   deformations of Gabor frames  are much more   difficult to
understand, because the frame operator  no longer  depends (norm-) continuously
on $\Lambda $ and a deformation 
may change the density of $\Lambda $ (which may affect significantly  the
spanning properties of $\cG (g, \Lambda ) $).

Perhaps the main difficulty is to find a suitable notion for
deformations that preserves Gabor frames. Except for linear
deformations and some  preliminary observations in ~\cite{CGN12,dG13}
this question is simply unexplored.  
In this paper we introduce a general  concept of deformation, which we call
\emph{Lipschitz} deformations. Lipschitz deformations include both  the jitter error and
linear deformations  as a special case. The precise definition is somewhat
technical and will be given  in Section~6.  For simplicity  we formulate a
representative special case of our main result.

\begin{theo}
\label{th_main_intro}
Let $g \in M^1(\Rdst)$ and  $\Lambda \subseteq \Rtdst$.
 Let $T_n: \Rtdst \to \Rtdst$ for  $n\in \bN $  be a sequence of  differentiable maps
 with Jacobian $DT_n$.    Assume that
\begin{align} \label{lipdefintro}
\sup_{z\in\Rtdst} \abs{DT_n(z)-I} \longrightarrow 0  \quad \text{ as }
n \to \infty \, . 
\end{align}

Then the following holds.
\begin{itemize}
\item[(a)] If $\mathcal{G}(g,\Lambda)$ is a frame, then
  $\mathcal{G}(g,T_n(\Lambda) )$ is a frame
for all sufficiently large $n$.
\item[(b)] If $\mathcal{G}(g,\Lambda)$ is a Riesz sequence, then
  $\mathcal{G}(g,T_n(\Lambda ))$ is a Riesz
sequence for all sufficiently large $n$.
\end{itemize}
\end{theo}

We would like to  emphasize that Theorem ~\ref{th_main_intro} is quite
 general. It deals with \emph{non-uniform} Gabor frames (not just
lattices)  under  \emph{nonlinear} deformations.  In particular,  Theorem~\ref{th_main_intro} implies
the main results of~\cite{feka04, asfeka13}.  The counterpart for
deformations of  Gabor Riesz sequences (item (b))   is
new even for linear deformations.

Condition~\eqref{lipdefintro} roughly states   that the    mutual 
distances between the points of $\Lambda$  are 
preserved locally under the deformation $T_n$. Our  main insight
was that the frame property of a deformed Gabor system $\cG (g,
T_n(\Lambda ))$ does not depend so much on the position or velocity of the
sequences $(T_n(\lambda ))_{ n\in \bN }$ for $\lambda \in \Lambda $, but on the relative distances
$|T_n(\lambda ) - T_n (\lambda ') |$ for $\lambda ,\lambda ' \in \Lambda
$. For an illustration see Example~\ref{ex_go_wrong}. 

As an application of Theorem~\ref{th_main_intro}, we derive a   \emph{non-uniform
Balian-Low theorem} (BLT). For this, we  recall that  
the  lower  Beurling density of a set  $\Lambda \subseteq \rdd $ is given by 
$$
D^-(\Lambda ) = \lim _{R\to \infty } \min _{z\in \rdd } \frac{\#
 \,  \Lambda \cap B_R(z)}{\mathrm{vol}(B_R(0))} \, ,
$$ 
and likewise the upper Beurling density  $D^+(\Lambda )$ (where the
minimum is  replaced by a supremum). The fundamental density theorem of Ramanathan and Steger~\cite{RS95}
asserts   that if  $\gab $ is a frame then  $D^-(\Lambda )\geq
1$. Analogously, if $\gab $ is a Riesz sequence, then $D^+(\Lambda ) \leq
1$~\cite{bacahela06-1}.  
The so-called  Balian-Low theorem (BLT)   is a stronger version
of the density theorem and asserts that for ``nice'' windows $g$  the 
inequalities in the density theorem are strict.    For the case when
$g\in M^1(\rd )$ and $\Lambda $ is a lattice,
the Balian-Low theorem is a consequence of ~\cite{feka04}. A  Balian-Low theorem  for non-uniform
Gabor frames  was open for a long time and was   proved only recently by  Ascensi, Feichtinger, and
Kaiblinger~\cite{asfeka13}. The corresponding statement for Gabor
Riesz sequences was open and is settled  here as an application of our
deformation theorem.   
We refer to Heil's detailed survey~\cite{heil07} of the  numerous
contributions to the density theorem for Gabor frames
after~\cite{RS95} and to ~\cite{CP06} for the Balian-Low theorem. 

As an immediate  consequence of Theorem~\ref{th_main_intro} we obtain the   following version
of the Balian-Low theorem  for non-uniform Gabor systems.  

\begin{coro}[Non-uniform Balian-Low Theorem]
Assume that $g\in M^1(\rd )$. 

(a) If  $\gab $ is a frame for $\lrd $,  then $D^-(\Lambda )>1$. 

(b) If  $\gab $ is a Riesz sequence in  $\lrd $,  then $D^+(\Lambda )<1$. 
\end{coro}

\begin{proof}
  We only prove the new statement (b), part (a) is similar~\cite{asfeka13}. Assume
  $\gab $ is a Riesz sequence, but that  $D^+(\Lambda ) =1$. Let
  $\alpha _n >1$ such that $\lim _{n\to \infty } \alpha _n  = 1$ and set
  $T_n z = \alpha _n z$.  Then the
  sequence $T_n$ satisfies condition \eqref{lipdefintro}. On the one
  hand, we have 
  $D^+(\alpha_n\Lambda) = \alpha _n ^{2d} >1$, and on the other hand,
  Theorem~\ref{th_main_intro} implies that $\cG (g,\alpha _n \Lambda
  )$ is a Riesz sequence for $n$ large enough. This is a contradiction
  to the density theorem, and thus the assumption $D^+(\Lambda ) =1$
  cannot hold. 
\end{proof}

\vs

The proof of  Theorem \ref{th_main_intro} does not come easily and is
technical. It 
combines methods from the theory of localized frames~\cite{fogr05,gr04}, the
stability of operators on $\ell ^p$-spaces~\cite{albakr08,sj95} and weak limit
techniques in the style of  Beurling~\cite{be89}.   We say that
$\Gamma \subseteq \rdd $ is a weak limit of translates of $\Lambda \subseteq 
\rdd$, if there exists a sequence  $(z_n)_{n\in \bN } \subseteq \rdd $,
such that $\Lambda+z_n\to \Gamma$ uniformly on compact sets. See Section~4 
for the precise definition and more details  on weak limits. 
 
We will prove the following characterization of non-uniform Gabor
frames ``without inequalities''.

\begin{theo} \label{th-char-frame}
Assume that $g\in M^1(\rd )  $ and $\Lambda \subseteq \rdd$. 
Then  $\gab $ is a frame for $\lrd $, \fif\ for every weak limit
$\Gamma $ of $\Lambda $ the map $f\to \big( \langle f, \pi (\gamma
)g\rangle \big) _{\gamma \in \Gamma }$ is one-to-one on $(M^1(\rd
))^*$. 
\end{theo}
The full statement with five equivalent conditions characterizing a
non-uniform Gabor frame will be given in Section~5,
Theorem~\ref{th_char_frame}. An analogous characterization of Gabor
Riesz sequences with weak limits is stated in
Theorem~\ref{th_char_riesz}. 

For the special case when  $\Lambda $ is a lattice, the above
characterization of Gabor frames without inequalities was already
proved in~\cite{gr07-2}. In the lattice case, the Gabor system $\gab $
possesses additional invariance properties that facilitate the
application of methods from operator algebras.  The
generalization of \cite{gr07-2} to non-uniform Gabor systems was
rather surprising for us and demands completely different methods. 

To make Theorem~\ref{th-char-frame} more plausible,
we make the analogy with Beurling's results on balayage in
Paley-Wiener space. Beurling~\cite{be89} characterized  the stability
of sampling  in the Paley-Wiener space of bandlimited functions $\{
f\in L^\infty (\rd ): \supp \hat{f} \subseteq S\}$ for a compact
spectrum $S\subseteq \rd $ in terms of sets of uniqueness for this
space. It is well-known that the frame property of a Gabor system
$\cG (g,\Lambda )  $ is equivalent to a sampling theorem for an associated
transform. Precisely, let $z\in \rdd \to V_gf(z) = \langle f, \pi
(z)g\rangle $ be the short-time Fourier transform,   for fixed non-zero $g\in
M^1(\rd )$ and $f\in (M^1(\rd ))^*$. Then  $\gab $ is a frame,
\fif\  $\Lambda $ is a set of sampling for the short-time \ft\ on
$(M^1)^*$. In this light, Theorem~\ref{th-char-frame} is the precise
analog of Beurling's theorem for bandlimited functions. 

One may therefore try to adapt  Beurling's methods
to Gabor frames and the sampling of short-time Fourier
transforms. Beurling's ideas have been used for many sampling problems
in complex analysis following the pioneering work of Seip on the Fock space 
\cite{Seip92a},\cite{Seip92b} and the Bergman space \cite{Seip}, see also
\cite{BMO} for a survey. A remarkable fact in
Theorem~\ref{th-char-frame}  is the absence of a complex structure (except 
when $g$ is a Gaussian). This explains why we have to use the
machinery of localized frames and the stability of operators in our
proof. We mention that Beurling's ideas  have been transferred to  a
few other contexts outside complex analysis, such as  sampling
theorems  with spherical harmonics in the sphere~\cite{MarOrt},  or,  more generally,
with eigenvectors of the Laplace operator in Riemannian
manifolds~\cite{OrtPri}. 

This article is organized as follows: 
In Section~2 we collect the main definitions from \tfa . In Section~3
we discuss \tf\ molecules and their $\ell ^p$-stability. Section~4 is
devoted to the details of Beurling's notion of weak convergence of
sets. In Section~5 we state and prove the full characterization of
non-uniform Gabor frames and Riesz sequences without inequalities. In
Section~6 we introduce the general concept of a Lipschitz deformation
of a set and prove the main properties. Finally, in Section~7 we state
and prove the main result of this paper, the general deformation
result. The appendix provides the technical proof of the stability of
\tf\ molecules.

\section{Preliminaries}
\subsection{General notation}
Throughout the article, $\abs{x} := (\abs{x_1}^2+\ldots+\abs{x_d}^2)^{1/2}$ denotes the Euclidean norm,
and $B_r(x)$ denotes the Euclidean ball. Given two functions $f,g:X \to [0,\infty)$, we say that
$f \lesssim g$ if there exists a constant $C>0$ such that $f(x) \leq C g(x)$, for all $x \in X$. We say that
$f \asymp g$ if $f \lesssim g$ and $g \lesssim f$.

\subsection{Sets of points}
A set $\Lambda \subseteq \Rdst$ is called \emph{relatively separated} if
\begin{align}
\label{eq_rel}
\rel(\Lambda) := \sup \{ \#(\Lambda \cap B_1(x)) : x \in \Rdst \} < \infty.
\end{align}
It is called \emph{separated} if
\begin{align}
\label{eq_sep}
\sep(\Lambda) := \inf \sett{\abs{\lambda-\lambda'}: \lambda \not = \lambda' \in \Lambda} > 0.
\end{align}
We say that $\Lambda$ is $\delta$-separated if $\sep(\Lambda) \geq \delta$.
A separated set is relatively separated and
\begin{align}
\label{eq_sep_relsep}
\rel(\Lambda) \lesssim \sep(\Lambda)^{-d}, \qquad \Lambda \subseteq \Rdst.
\end{align}
Relatively separated sets are finite unions of separated sets.

The \emph{hole} of a set $\Lambda \subseteq \Rdst$ is defined as
\begin{align}
\label{eq_gap}
\rho(\Lambda) := \sup_{x\in\Rdst} \inf_{\lambda \in \Lambda} \abs{x-\lambda}.
\end{align}
A sequence $\Lambda$ is called \emph{relatively dense} if $\rho(\Lambda) < 
\infty$. Equivalently, $\Lambda$ is relatively
dense if there exists $R>0$ such that
\begin{align*}
\Rdst = \bigcup_{\lambda \in \Lambda} B_R(\lambda).
\end{align*}
In terms of the Beurling densities defined in the introduction, a set $\Lambda$ is relatively separated
if and only if $D^{+}(\Lambda)<\infty$ and it is relatively dense if and only if $D^{-}(\Lambda)>0$.

\subsection{Amalgam spaces}
\label{sec_am}
The \emph{amalgam space} $W(L^\infty,L^1)(\Rdst)$ consists of all functions $f \in L^\infty(\Rdst)$
such that
\begin{align*}
\norm{f}_{W(L^\infty,L^1)} := \int_{\Rdst} \norm{f}_{L^\infty(B_1(x))} dx
\asymp
\sum_{k \in \Zdst} \norm{f}_{L^\infty([0,1]^d+k)} <\infty.
\end{align*}
The space $C_0(\Rdst)$ consists of all continuous functions $f: \Rdst \to \bC$
such that $\lim_{x \longrightarrow \infty} f(x) = 0$, consequently 
the (closed) subspace of $W(L^\infty,L^1)(\Rdst)$ consisting of continuous functions is 
$W(C_0,L^1)(\Rdst)$. 
 This space will be used as a convenient
collection of test functions.

We will repeatedly use the following sampling inequality:
 \emph{Assume that $f\in W(C_0,L^1)(\Rdst)$ and $\Lambda \subseteq
 \rd $ is relatively separated, then }
\begin{equation}
  \label{eq:c7}
  \sum _{\lambda \in \Lambda } |f(\lambda )| \lesssim \rel
  (\Lambda ) \, \| f\|_{W(L^\infty , L^1)}\, .
\end{equation}

The dual space of $W(C_0,L^1)(\Rdst)$ will be denoted $\wmes(\Rdst)$. It consists of all the complex-valued Borel
measures $\mu: \mathcal{B}(\Rdst) \to \bC$ such that
\begin{align*}
\norm{\mu}_\wmes := \sup_{x \in \Rdst} \norm{\mu}_{B_1(x)} 
= \sup_{x \in \Rdst} \abs{\mu}(B_1(x)) < \infty.
\end{align*}
For the general theory of  Wiener amalgam spaces we refer to  \cite{fe83}. 

\subsection{Time-frequency analysis}
The \emph{time-frequency shifts} of a function $f: \Rdst \to \bC$ are
\begin{align*}
\pi(z)f(t) := e^{2\pi i \xi t} f(t-x), \qquad z=(x,\xi) \in \Rdst\times \Rdst,
t \in \Rdst.
\end{align*}
These operators satisfy the commutation relations
\begin{align}
\label{eq_comp_tf}
\pi(x,\xi) \pi(x',\xi') = e^{-2\pi i \xi' x} \pi(x+x', \xi+\xi'),
\qquad (x,\xi), (x',\xi') \in \Rdst \times \Rdst.
\end{align}
Given a non-zero Schwartz function $g \in \mathcal{S}(\Rdst)$, the \emph{short-time Fourier transform} of a 
distribution $f \in \mathcal{S}'(\Rdst)$ with respect to the window $g$ is defined as
\begin{align}
\label{eq_def_stft}
V_g f(z) := \ip{f}{\pi(z) g}, \qquad z \in \Rtdst.
\end{align}
For $\norm{g}_2=1$ the short-time Fourier transform  is an isometry: 
\begin{align}
\label{eq_stft_l2}
\norm{V_g f}_{\LtRtd}=\norm{f}_{\LtRd},
\qquad f \in \LtRd.
\end{align}
The commutation rule \eqref{eq_comp_tf} implies the covariance
property of  the short-time Fourier transform: 
\begin{align*}
V_g (\pi(x,\xi) f)(x',\xi') = e^{-2\pi i x(\xi '-\xi )} V_g f(x'-x,\xi'-\xi),
\qquad (x,\xi), (x',\xi') \in \Rdst \times \Rdst.
\end{align*}
In particular,
\begin{align}
\label{eq_tf_stft}
\abs{V_g \pi(z) f} = \abs{V_g f(\cdot-z)},
\qquad z \in \Rtdst.
\end{align}

We then define the \emph{modulation spaces} as follows: fix a non-zero
$g\in \cS (\rd )$ and let 
\begin{align}
\label{eq_def_mp}
M^p(\Rdst) := \set{f \in \mathcal{S}'(\Rdst)}{V_g f \in L^p(\Rtdst)},
\qquad 1 \leq p \leq \infty,
\end{align}
endowed with  the norm $\norm{f}_{M^p} := \norm{V_g f}_{L^p}$. Different choices of non-zero windows
$g \in \mathcal{S}(\Rdst)$ yield the same space with equivalent norms, see 
\cite{fe89}. We note that for $g\in M^1(\rd )$ and $f\in M^p(\rd )$,
$1\leq p \leq \infty $, the short-time Fourier transform  $V_gf$ is a continuous function, we
may therefore argue safely with the pointwise values of $V_gf$. 

The space $M^1(\Rdst)$, known as the Feichtinger algebra, plays a central role.
It can also  be characterized as
\begin{align*}
M^1(\Rdst) := \set{f \in L^2(\Rdst)}{V_f f \in L^1(\Rtdst)}.
\end{align*}
 The  \modsp\  $\modzero(\Rdst)$ is defined  as the closure of the
 Schwartz-class with respect to the norm $\| \cdot \|_{M^\infty
 }$. Then $M^0(\rd )$ is a closed subspace of   $M^\infty(\Rdst)$ and
 can also be characterized as
\begin{align*}
\modzero(\Rdst) = \set{f \in M^\infty(\Rdst)}{V_g f \in C_0(\Rtdst)}.
\end{align*}
The duality of \modsp s is  similar to  sequence spaces; we have 
$M^0(\rd )^* = M^1(\rd )$ and $M^1(\rd )^* = M^\infty (\rd )$ with
respect to the duality $\langle f, h \rangle := \langle V_g f , V_g h
\rangle $. 

In this article we consider a fixed function $g \in M^1(\Rdst)$ and will be mostly concerned with $M^1(\Rdst)$, its
dual space $M^\infty(\Rdst)$, and $M^2(\Rdst)=L^2(\Rdst)$. The weak* topology in $M^\infty(\Rdst)$ will be denoted by
$\sigma(M^\infty,M^1)$ and  the weak* topology on $M^1(\Rdst)$ by
$\sigma(M^1,\modzero)$. Hence, a sequence $\sett{f_k:k \geq 1}
\subseteq M^\infty(\Rdst)$ converges to $f \in 
M^\infty(\Rdst)$ in $\sigma(M^\infty,M^1)$ if and only if for every $h \in M^1(\Rdst)$:
$\ip{f_k}{h} \longrightarrow \ip{f}{h}$.

We mention the following facts that will be used repeatedly
(see for example \cite[Theorem 4.1]{fegr89} and \cite[Proposition 12.1.11]{gr01}).
\begin{lemma}
\label{lemma_stft}
Let $g \in M^1(\Rdst)$ be nonzero. Then the following hold true.
\begin{itemize}
\item[(a)] If $f \in M^1(\Rdst)$, then $V_g f \in W(C_0,L^1)(\Rtdst)$.
\item[(b)] Let $\sett{f_k:k \geq 1} \subseteq M^\infty(\Rdst)$ be a bounded sequence
and $f \in M^\infty(\Rdst)$. Then
$f_k \longrightarrow f$ in $\sigma(M^\infty,M^1)$
if and only if
$V_g f_k \longrightarrow V_g f$ uniformly on compact sets.
\item[(c)] Let $\sett{f_k:k \geq 1} \subseteq M^1(\Rdst)$ be a bounded sequence
and $f \in M^1(\Rdst)$. Then $f_k \longrightarrow f$ in $\sigma(M^1,\modzero)$
if and only if $V_g f_k \longrightarrow V_g f$ uniformly on compact sets.
\end{itemize}
\end{lemma}

In particular, if $f_n \to f $ in $\sigma (M^\infty , M^1)$ and $z_n
\to z \in \rdd $, then $V_gf_n (z_n) \to V_gf(z)$. 

\subsection{Analysis and synthesis maps}
\label{sec_maps}
Given $g \in M^1(\Rdst)$ and a relatively separated set $\Lambda \subseteq 
\Rtdst$, consider the
\emph{analysis operator}  and the  \emph{synthesis operator} that are  formally defined as
\begin{align*}
&C_{g, \Lambda} f := \left( \ip{f}{\pi(\lambda) g} \right)_{\lambda \in \Lambda}, \qquad f \in M^\infty(\Rdst),
\\
&C^*_{g, \Lambda} c := \sum_{\lambda \in \Lambda} c_\lambda \pi(\lambda) g,
\qquad c \in \ell^\infty(\Lambda).
\end{align*}
These maps are bounded between $M^p$ and $\ell^p$ spaces~\cite[Cor. 12.1.12]{gr01}
 with estimates 
\begin{align*}
&\norm{C_{g, \Lambda} f}_{\ell^p}
\lesssim \rel(\Lambda) \norm{g}_{M^1} \norm{f}_{M^p},
\\
&\norm{C^*_{g, \Lambda} c}_{M^p}
\lesssim  \rel(\Lambda) \norm{g}_{M^1} \norm{c}_{\ell^p}.
\end{align*}
The implicit constants in the last estimates are independent of $p\in [1,\infty ]$.

For $z=(x,\xi) \in \Rtdst$, the \emph{twisted shift} is the operator
$\twistshift(z): \ell ^\infty (\Lambda )  \to \ell^\infty (\Lambda+z )
$ given by
\begin{align*}
(\twistshift(z) c)_{\lambda+z} := e^{-2\pi i x \lambda_2} c_\lambda, \qquad \lambda=(\lambda_1,\lambda_2) \in \Lambda.
\end{align*}
As a consequence of the commutation relations~\eqref{eq_comp_tf}, the
analysis and synthesis operators satisfy the covariance property 
\begin{align}
\label{eq_cov_c}
\pi(z) C^*_{g,\Lambda}=C^*_{g,\Lambda+z} \twistshift(z)
\mbox{ and }
e^{2\pi i x \xi}C_{g,\Lambda}\pi(-z)=e^{-2\pi i x
  \xi}\twistshift(-z)C_{g,\Lambda+z} \, 
\end{align}
 for $z=(x,\xi)\in\Rdst\times\Rdst$. 

A Gabor system $\mathcal{G}(g,\Lambda)$ is a \emph{frame} if and only if $C_{g, \Lambda}: \LtRd \to \ell^2(\Lambda)$ is
bounded below,  and $\mathcal{G}(g,\Lambda)$  is a \emph{Riesz sequence} if and only if $C^*_{g, \Lambda}: \ell^2(\Lambda) \to \LtRd$ is bounded
below. As the following lemma shows, each of these conditions implies a restriction of the geometry of the set
$\Lambda$.

\begin{lemma}
Let $g \in \LtRd$ and let $\Lambda \subseteq \Rtdst$ be a set. Then the following holds.
\label{lemma_set_must_be}
\begin{itemize}
\item[(a)] If $\mathcal{G}(g,\Lambda)$ is a frame, then $\Lambda$ is relatively 
separated and relatively dense.
\item[(b)] If $\mathcal{G}(g,\Lambda)$ is a Riesz sequence, then $\Lambda$ is separated.
\end{itemize} 
\end{lemma}
\begin{proof}
For part (a) see for example \cite[Theorem 1.1]{chdehe99}. For part (b), suppose that
$\Lambda$ is not separated. Then there exist two sequences $\sett{\lambda_n: n \geq 1},
\sett{\gamma_n: n \geq 1} \subseteq \Lambda$ with $\lambda_n \not= \gamma_n$
such that $\abs{\lambda_n-\gamma_n} \longrightarrow 0$.
Hence we derive the following contradiction:
$\sqrt{2}=\norm{\delta_{\lambda_n}-\delta_{\gamma_n}}_{\ell^2(\Lambda)} \asymp
\norm{\pi(\lambda_n)g-\pi(\gamma_n)g}_{L^2(\Rdst)} \longrightarrow 0$.
\end{proof}

We extend the previous terminology to other values of $p \in [1,\infty]$. We say that $\mathcal{G}(g,\Lambda)$ is a
$p$-frame
for $M^p(\Rdst)$ if $C_{g, \Lambda}: M^p(\Rdst) \to \ell^p(\Lambda)$ is bounded below, and that $\mathcal{G}(g,\Lambda)$
is a $p$-Riesz sequence within $M^p(\Rdst)$ if $C^*_{g, \Lambda}: \ell^p(\Lambda) \to M^p(\Rdst)$ is bounded below.
Since boundedness below and left invertibility are different properties outside the context of Hilbert spaces, there
are other reasonable definitions of frames and Riesz sequences for $M^p$. This is largely immaterial for Gabor frames
with $g \in M^1$, since the theory of localized frames asserts that when such a system is a frame for $L^2$,
then it is a frame for all $M^p$ and moreover the operator $C_{g, \Lambda}:M^p \to \ell^p$ is left invertible
\cite{gr04, fogr05, bacahela06, bacahela06-1}.
Similar statements apply to Riesz sequences.

\section{Stability of time-frequency molecules}
We say that $\sett{f_\lambda: \lambda \in\Lambda} \subseteq \LtRd$ is a 
set  of time-frequency molecules,  if $\Lambda \subseteq \Rtdst$ is
a relatively separated set and there exists a non-zero $g\in M^1(\rd )$ and an envelope function  $\Phi
\in W(L^\infty,L^1)(\Rtdst)$ such that
\begin{align}
\label{eq_env_mol}
\abs{\stft_g f_\lambda (z)} \leq \Phi(z-\lambda),
\qquad \mbox{a.e. } z \in \Rdst, \lambda \in \Lambda . 
\end{align}
If \eqref{eq_env_mol} holds for some $g\in M^1(\rd )$, then it holds
for all $g\in M^1(\rd )$ (with an  envelope depending on $g$). 
\begin{rem}
\label{rem_gab_mol}
{\rm 
Every Gabor system  $\mathcal{G}(g,\Lambda)$ with window  $g \in M^1(\Rdst)$
and  a relatively separated set $\Lambda \subseteq \Rtdst$  is  a set of
time-frequency molecules.  
In this case the envelope can be chosen to be  $\Phi = \abs{V_gg}$,
which  belongs to $W(L^\infty,L^1)(\Rtdst)$ by  Lemma \ref{lemma_stft}.}
\end{rem}

 The following stability result will be one of our main technical tools.
\begin{theo}
\label{th_main_mol} Let $\sett{f_\lambda: \lambda \in \Lambda}$ be a set of 
time-frequency molecules. Then the following holds.
\begin{itemize}
\item[(a)] Assume that
\begin{align}
\label{eq_main_1}
\norm{f}_{M^p} \asymp \norm{(\ip{f}{f_\lambda})_{\lambda\in\Lambda}}_p,
\qquad \forall \, f \in M^p(\Rdst),
\end{align}
holds for some $1 \leq p \leq \infty$. Then \eqref{eq_main_1} holds
for all $1 \leq p \leq \infty$. In other words, if $\sett{ f_\lambda :
  \lambda \in \Lambda } $ is a $p$-frame for $M^p(\rd )$ for some
$p\in [1,\infty]$, then it is a $p$-frame for $M^p(\rd )$ for all $p\in [1,\infty]$. 

\item[(b)] Assume that 
\begin{align}
\label{eq_main_2}
\norm{\sum_{\lambda \in \Lambda} c_\lambda f_\lambda}_{M^p} \asymp \norm{c}_p,
\qquad c \in \ell^p(\Lambda),
\end{align}
holds for some $1 \leq p \leq \infty$. Then \eqref{eq_main_2} holds for all $1 \leq p \leq \infty$.
\end{itemize}
\end{theo}
The result is similar in spirit to other results in the literature \cite{su07-5, albakr08, shsu09, te10, su10-2}, but
 none of these is  directly  applicable to our setting. We postpone
 the proof of Theorem 
\ref{th_main_mol} to the appendix, so as not to interrupt the natural
flow of the article. As in the cited  references,
the proof elaborates on Sj\"ostrand's Wiener-type lemma \cite{sj95}. 

As a special case  of Theorem \ref{th_main_mol} we record  the
following corollary. 
\begin{coro}
\label{coro_main_mol}
Let $g \in M^1(\Rdst)$ and let $\Lambda \subseteq \Rtdst$ be a relatively separated set. Then the following holds.
\begin{itemize}
\item[(a)] If $\mathcal{G}(g,\Lambda)$ is a $p$-frame for $M^p(\Rdst)$ for some $p \in [1, \infty]$, then it is a
$p$-frame for $M^p(\Rdst)$  for all $p \in [1,\infty]$.
\item[(b)] If $\mathcal{G}(g,\Lambda)$ is a $p$-Riesz sequence in
  $M^p(\Rdst)$ for some $p \in [1, \infty]$, then 
it is $p$-Riesz sequence in $M^p(\Rdst)$ for all $p \in [1,\infty]$.
\end{itemize}
\end{coro}

The space $M^1(\rd )$ is the largest space of windows for which the
corollary holds. Under a stronger condition on  $g$,  statement (a)  was
already derived in~\cite{albakr08}, the general case was left open.  

\section{Weak convergence}
\subsection{Convergence of sets}
The Hausdorff distance between two sets $X,Y \subseteq \Rdst$ is defined as
\begin{align*}
d_H(X,Y) := \inf \sett{\varepsilon>0: X \subseteq Y+ B_\varepsilon(0),
Y \subseteq X+ B_\varepsilon(0)}.
\end{align*}
Note that $d_H(X,Y)=0$ if and only if $\overline{X} = \overline{Y}$.

Let $\Lambda \subseteq \Rdst$ be a set. A sequence
$\{\Lambda_n: n \geq 1\}$ of subsets of $\Rdst$ \emph{converges
  weakly} to $\Lambda$, in short $\Lambda _n \weakconv \Lambda$,  if 
\begin{align}
\label{eq_weak_conv}
d_H \big( (\Lambda_n \cap \bar{B}_R(z)) \cup \partial \bar{B}_R(z))
, (\Lambda  \cap \bar{B}_R(z) ) \cup \partial \bar{B}_R(z))
\big) \to 0, \qquad \forall z\in \Rdst , R>0 \,  . 
\end{align}
(To understand  the role of the boundary of the ball in the definition, consider the following example
in dimension $d=1$: $\Lambda_n:=\sett{1+1/n}$, $\Lambda:=\sett{1}$ and $B_R(z)=[0,1]$.) 

 The following lemma provides
an  alternative description of weak convergence.
\begin{lemma}
\label{lemma_alt_weak}
Let $\Lambda \subseteq \Rdst$ and $\Lambda_n \subseteq \Rdst, n \geq 1$ be sets.
Then $\Lambda_n \weakconv \Lambda$ if and only if for every $R>0$ and $\varepsilon>0$
there exists $n_0 \in \Nst$ such that for all $n \geq n_0$
\begin{align*}
\Lambda \cap B_R(0) \subseteq \Lambda_n + B_\varepsilon(0) \quad
\text{ and }  \quad \Lambda_n \cap B_R(0) \subseteq \Lambda + B_\varepsilon(0).
\end{align*}
\end{lemma}

The following  consequence of Lemma \ref{lemma_alt_weak} is often
useful to identify weak limits. 
\begin{lemma}
\label{lemma_weak_inc}
Let $\Lambda_n \weakconv \Lambda$ and $\Gamma_n \weakconv
\Gamma$. Suppose that for every $R>0$ and $\varepsilon>0$
there exists $n_0 \in \Nst$ such that for all $n \geq n_0$
\begin{align*}
&\Lambda_n \cap B_R(0) \subseteq \Gamma_n + B_\varepsilon(0).
\end{align*}
Then $\overline{\Lambda} \subseteq \overline{\Gamma}$.
\end{lemma}

The notion of weak convergence will be a technical tool in the proofs
of   deformation results. 

\subsection{Measures and compactness}
In this section we explain how the weak convergence of sets can be understood by the convergence of some
associated measures. First we note the following semicontinuity property, that follows directly from Lemma
\ref{lemma_alt_weak}.

\begin{lemma}
\label{lemma_mes_supp_a}
Let $\sett{\mu_n : n \geq 1} \subset \wmes(\Rdst)$ be a sequence of measures that converges to a measure
$\mu \in \wmes(\Rdst)$ in the $\sigma(\wmes, \wtest)$ topology. Suppose that $\supp(\mu_n) \subseteq \Lambda_n$
and that $\Lambda_n \weakconv \Lambda$. Then $\supp(\mu) \subseteq \overline{\Lambda}$.
\end{lemma}
The example $\mu_n=\tfrac{1}{n}\delta$, $\mu = 0$ shows that in Lemma \ref{lemma_mes_supp_a} the inclusions cannot 
in general be improved to equalities. Such improvement is however possible for certain classes of measures. A 
Borel measure $\mu$ is called \emph{natural-valued} if for all Borel sets $E$ the value
$\mu(E)$ is a non-negative integer or infinity. For these measures the following holds.
\begin{lemma}
\label{lemma_mes_supp_b}
Let $\sett{\mu_n : n \geq 1} \subset \wmes(\Rdst)$ be a sequence of natural-valued measures that converges to a measure
$\mu \in \wmes(\Rdst)$ in the $\sigma(\wmes, \wtest)$ topology. Then 
$\supp(\mu_n) \weakconv \supp(\mu)$.
\end{lemma}
The proof of Lemma \ref{lemma_mes_supp_b} is elementary and therefore we skip it.
Lemma \ref{lemma_mes_supp_b} is useful to deduce properties of weak convergence of sets from properties of convergence of
measures, as we now show. For a set $\Lambda \subseteq \Rdst$, let us consider the 
natural-valued measure
\begin{align}
\label{eq_shah}
\Shah_\Lambda := \sum_{\lambda \in \Lambda} \delta_\lambda.
\end{align}
One can readily verify that $\Lambda$ is relatively separated if and only if $\Shah_\Lambda \in \wmes(\Rdst)$ and
moreover,
\begin{align}
\label{eq_mes_am}
&\norm{\Shah_\Lambda}_\wmes \asymp \rel(\Lambda).
\end{align}

For sequences of sets $\sett{\Lambda_n: n\geq 1}$ with uniform
separation, i.e., 
\begin{align*}
\inf_n \sep (\Lambda _n) = \inf\{\abs{\lambda - \lambda'}: \lambda \not= \lambda', \lambda,\lambda' \in \Lambda_n, n
\geq 1\} >0, 
\end{align*}
the convergence  $\Lambda_n \weakconv \Lambda$ is equivalent to the
convergence $\Shah_{\Lambda_n} \to  \Shah_\Lambda$ 
in \\  $\sigma( \wmes, \wtest)$. For sequences without uniform separation
the situation is slightly more technical 
because of possible multiplicities in the limit set.
\begin{lemma}
\label{lemma_compactness}
Let $\sett{\Lambda_n: n \geq 1}$ be a sequence of relatively separated 
sets in  $\Rdst$. Then the following hold.
\begin{itemize}
\item[(a)] If $\Shah_{\Lambda_n} \longrightarrow \mu$ in
  $\sigma(\wmes, \wtest)$  
for some measure $\mu \in \wmes$, then
$\sup_n \rel(\Lambda_n) < \infty$ and $\Lambda_n \weakconv \Lambda := \supp(\mu)$.
\item[(b)] If $\limsup_n \rel(\Lambda_n) < \infty$, then there exists a subsequence
$\sett{\Lambda_{n_k}: k \geq 1}$ that converges weakly to a relatively separated 
set.
\item[(c)] If $\limsup_n \rel(\Lambda_n) < \infty$ and $\Lambda_n
  \weakconv \Lambda$  for some set
$\Lambda \subseteq \Rdst$,  then $\Lambda$ is relatively separated (and in particular closed).
\end{itemize}
\end{lemma}
The lemma follows easily from Lemma \ref{lemma_mes_supp_b},
\eqref{eq_mes_am} and the weak$^*$-compactness of the ball of 
$\wmes$, and hence we do not prove it. We remark that the limiting
measure $\mu$ in the lemma is not necessarily 
$\Shah_\Lambda$. For example, if $d=1$ and $\Lambda_n := \sett{0,1/n,1,1+1/n,1-1/n}$, then $\Shah_{\Lambda_n} 
\longrightarrow 2 \delta_0 + 3 \delta_1$. The measure $\mu$ in (a) can be shown to be natural-valued, and therefore
we can interpret it as representing a set with multiplicities. 

The following lemma provides a version of \eqref{eq_mes_am} for linear
combinations of point measures.
\begin{lemma}
\label{lemma_norms_mes}
Let $\Lambda \subseteq \Rdst$ be a relatively separated set and consider a measure
\begin{align*}
\mu := \sum_{\lambda \in \Lambda} c_\lambda \delta_\lambda 
\end{align*}
with coefficients  $c_\lambda \in \bC$. Then
\begin{align*}
&\norm{\mu} = \abs{\mu}(\Rdst) = \norm{c}_1,
\\
&\norm{c}_\infty \leq \norm{\mu}_{\wmes} \lesssim \rel(\Lambda) \norm{c}_\infty.
\end{align*}
\end{lemma}
\begin{proof}
The identity  $\abs{\mu}(\Rdst) = \norm{c}_1$ is elementary. The estimate for $\norm{\mu}_{\wmes}$
follows from the fact that, for all $\lambda \in \Lambda$, 
$\abs{c_\lambda} \delta_\lambda \leq \abs{\mu} \leq \norm{c}_\infty \Shah_\Lambda$, where
$\Shah_\Lambda$ is defined by~\eqref{eq_shah}.
\end{proof}

\section{Gabor Frames and  Gabor Riesz Sequences without Inequalities}
\label{sec_char}
As a first step towards the main results, we  characterize frames and Riesz bases in terms of uniqueness properties for certain limit 
sequences. The corresponding results for lattices have been derived by
different methods in \cite{gr07-2}. For the proofs we  combine Theorem
\ref{th_main_mol} with Beurling's methods~\cite[p.351-365]{be89}. 
              
For a relatively separated set $\Lambda \subseteq \Rtdst$, let $W(\Lambda)$ be the set of weak limits of the 
translated sets $\Lambda+z, z\in \Rtdst$,  i.e.,  $\Lambdastar \in
W(\Lambda)$ if there exists a sequence 
$\sett{z_n:n \in \Nst}$ such that $\Lambda+z_n \weakconv \Lambdastar$. It is easy to see that then $\Gamma$ is
always relatively separated.
When $\Lambda$ is a lattice, i.e.,  $\Lambda=A \Ztdst$ for an invertible real-valued $2d\times2d$-matrix $A$, then  
$W(\Lambda)$ consists only of translates of $\Lambda$. 
 
Throughout this section we use repeatedly the following special case
of Lemma \ref{lemma_compactness}(b,c): given a
relatively separated set $\Lambda 
\subseteq \Rtdst$ and any sequence of points $\sett{z_n: n \geq 1} \subseteq \Rtdst$, there is a subsequence
$\sett{z_{n_k}: k \geq 1}$ and a relatively separated set $\Lambdastar \subseteq \Rtdst$ such that $\Lambda + z_{n_k}
\weakconv \Lambdastar$.

\subsection{Characterization of frames}
 In this section we characterize the frame property of Gabor systems in terms of
the sets in $W(\Lambda)$.

\begin{theo}
\label{th_char_frame} 
Assume that $g\in M^1(\Rdst)$ and that $\Lambda \subseteq \Rtdst $ is
relatively separated. Then the following are equivalent.
\begin{itemize}
\item[(i)] $\gab$ is a frame for $\ltrd$. 

\item[(ii)]  $\gab$ is a $p$-frame for $M^p(\Rdst)$ for some  $p \in
  [1,\infty]$ (for all $p\in [1,\infty ]$).

\item[(iii)] $\gab$ is an $\infty$-frame for $M^{\infty}(\Rdst)$. 

\item[(iv)] $C_{g,\Lambda }^*$ is surjective from     $\ell  ^1(\Lambda
    )$   onto $M^1(\rd )$. 

\item[(v)] $C_{g,\Lambdastar}$ is bounded below on $M^\infty(\Rdst) $ for 
every weak limit $\Lambdastar \in W(\Lambda )$. 

\item[(vi)] $C_{g,\Lambdastar}$ is one-to-one on $M^\infty(\Rdst) $ for every 
weak limit $\Lambdastar \in W(\Lambda)$.
\end{itemize}
\end{theo}
\begin{rem}
{\rm 1. When $\Lambda$ is a lattice, then $W(\Lambda)$ consists only of translates of
$\Lambda$. In this case,  Theorem \ref{th_char_frame} reduces to main result 
in \cite{gr07-2}. }  \\
{\rm 2. For related work in the context of sampling measures see also
  \cite{as13}.
}
\end{rem}
\begin{proof}
The equivalence of (i), (ii) and (iii) follows immediately from Corollary \ref{coro_main_mol}.

In the sequel  we will use several times the following version of the
closed range theorem~\cite[p.~166]{conway90}: Let $T:X\to Y$ be a
bounded operator  between two Banach spaces $X$ and $Y$. Then $T$ is
onto $Y$,  \fif\ $T^*: Y^* \to X^*$ is one-to-one on $Y^*$ and has closed range in
$X^*$, \fif\  $T^*$ is  bounded below.

Conditions  (iii) and   (iv) are equivalent by applying the closed
range theorem to the synthesis operator $C^*_{g,\Lambda }$ on $\ell
^1(\Lambda )$.

For the remaining equivalences we adapt Beurling's methods. 

{\bf (iv) $\Rightarrow$ (v)}. Consider a convergent sequence of translates
$\Lambda - z_n \weakconv \Lambdastar$. Since $C_{g,\Lambda }^*$
maps $\ell ^1(\Lambda )$ onto $M^1(\rd )$, because of \eqref{eq_cov_c}
and the open mapping theorem,
the synthesis operators
$C_{g,\Lambda -z_n}^*$ are also onto $M^1(\Rdst)$ with bounds on preimages
independent of $n$. Thus for every $f\in M^1(\rd )$ 
there exist sequences $\sett{c^n_\lambda}_{\lambda \in \Lambda - z_n}$ with $\norm{c^n}_1
\lesssim 1$ such that
\begin{align*}
f = \sum_{\lambda \in \Lambda - z_n} c^n_\lambda \pi(\lambda) g,
\end{align*}
with convergence in $M^1(\Rdst)$.

Consider the measures $\mu_n := \sum_{\lambda \in \Lambda - z_n} c^n_\lambda \delta_{\lambda}$. 
Note that $\norm{\mu_n} = \norm{c^n}_1 \lesssim 1$. By passing to a
subsequence we may assume that $\mu_n \longrightarrow \mu$ in
$\sigma(\mathcal{M},C_0)$ 
for some measure $\mu \in \mathcal{M}(\Rtdst)$.

By assumption  $\supp(\mu_n) \subseteq \Lambda - z_n$, $\Lambda - z_n
\weakconv \Lambdastar$, 
and $\Lambdastar$ is  relatively separated and thus closed.   
It follows from Lemma \ref{lemma_mes_supp_a} that  $\supp(\mu) \subseteq \Lambdastar$. Hence,
\begin{align*}
\mu = \sum_{\lambda \in \Lambdastar} c_\lambda \delta_{\lambda} 
\end{align*}
for some sequence $c$. In addition,
$\norm{c}_1 = \norm{\mu} \leq \liminf_n \norm{\mu_n} \lesssim 1$.
Let $f' := \sum_{\lambda \in \Lambdastar} c_\lambda \pi(\lambda)
g$. This is well-defined in $M^1(\rd )$,  because $c \in
\ell^1(\Lambdastar)$. Let $z\in\Rtdst$. Since by Lemma \ref{lemma_stft}
$V_g \pi(z) g \in \wtest(\Rtdst) \subseteq C_0(\Rtdst)$ we can compute
\begin{align*}
\ip{f}{\pi(z)g} &= \sum_{\lambda \in \Lambda - z_n} c^n_\lambda \overline{V_g \pi(z) g(\lambda)}
\\
&=\int_{\Rtdst} \overline{V_g \pi(z) g} \,d\mu_n
\longrightarrow
\int_{\Rtdst} \overline{V_g \pi(z) g} \,d\mu=
\ip{f'}{\pi(z)g}.
\end{align*}
(Here, the interchange of summation and integration is justified because $c$ and $c^n$ are summable.)
Hence $f=f'$ and thus $C^*_{g,\Gamma}: \ell^1(\Gamma) \to M^1(\Rdst)$ is surjective. By
duality $C_{g,\Gamma }$ is one-to-one from $M^\infty (\rd )$ to $\ell
^\infty (\Gamma )$ and has closed range, whence $C_{g,\Gamma }$ is
bounded below on $M^\infty (\infty )$.  

{\bf (v) $\Rightarrow$ (vi)} is clear.

{\bf (vi) $\Rightarrow$ (iii).} Suppose $\gab$ is not an $\infty$-frame for $M^\infty(\Rdst)$. Then
there exists a sequence of functions $\sett{f_n: n \geq 1} \subset M^\infty(\Rdst)$ such that
$\norm{V_g f_n}_\infty=1$ and $\sup_{\lambda \in \Lambda} \abs{V_g f_n (\lambda)} \longrightarrow 0$. Let $z_n \in
\Rtdst$ be such that $\abs{V_g f_n(z_n)} \geq 1/2$ and consider $h_n := \pi(-z_n)f_n$. By passing to a subsequence
we may assume that $h_n \longrightarrow h$ in $\sigma(M^\infty, M^1)$ for some $h \in M^\infty(\Rdst)$, and
that $\Lambda - z_n \weakconv \Gamma$ for some relatively separated $\Gamma$. Since
$\abs{V_g h_n (0)}=\abs{V_g f_n (z_n)} \geq 1/2$
by~\eqref{eq_tf_stft},  it follows from Lemma \ref{lemma_stft} (b)
that $h \not= 0$. Given $\gamma
\in \Gamma$, there exists  a sequence $\sett{\lambda_n: n \geq 1}
\subseteq \Lambda$ such that $\lambda_n - z_n \longrightarrow \gamma$. Since,
by Lemma \ref{lemma_stft}, $V_g h_n \longrightarrow V_g h$ uniformly on compact sets, 
we can use \eqref{eq_tf_stft} to obtain that
\begin{align*}
\abs{V_g h(\gamma)} = \lim_n \abs{V_g h_n(\lambda_n - z_n)} = \lim_n \abs{V_g f_n(\lambda_n)} = 0.
\end{align*}
As $\gamma \in \Gamma $ is arbitrary, this contradicts (vi).
\end{proof}

Although Theorem~\ref{th_char_frame} seems to be purely qualitative,
it can be used to derive quantitative estimates for Gabor frames. 
We fix a non-zero window  $g$ in $ M^1(\rd )$ and assume that
$\|g\|_2=1$. We measure the modulation space norms with respect to
this window  by $\|f\|_{M^p} = \|V_g f\|_p$ and observe that 
the isometry property of the short-time Fourier transform extends to
$M^\infty (\rd )$ as follows: if $f\in M^\infty (\rd ) $ and $h\in
M^1(\rd )$, then 
\begin{equation}
  \label{eq:c10}
  \langle f, h\rangle = \intrdd V_gf(z) \overline{V_gh(z)} \, dz =
  \langle V_gf, V_gh\rangle.
\end{equation}
 For $\delta >0$, we  define the
$M^1$-modulus of continuity of $g$ as 
\begin{equation}
\label{eq:c9}
\omega _\delta (g)_{M^1} = \sup_{\stackrel{z,w \in \rdd}{|z-w|\leq \delta }}
\|\pi (z)g- \pi (w)g\|_{M^1}
=\sup_{\stackrel{z,w \in \rdd}{|z-w|\leq \delta }} \norm{V_g(\pi (z)g- \pi (w)g)}_{L^1}.
\end{equation}
It is easy to verify that $\lim _{\delta \to 0+}   \omega _\delta
(g)_{M^1} = 0$, because time-frequency shifts are continuous on $M^1(\rd )$. 

Then we deduce the following quantitative condition for Gabor frames
from Theorem~\ref{th_char_frame}. 

\begin{coro}
  For $g\in M^1(\rd )$ with $\norm{g}_2=1$ choose $\delta >0$ so that $  \omega
  _\delta (g)_{M^1} <1$.

If  $\Lambda \subseteq \rdd $ is relatively
  separated and  $\rho (\Lambda ) \leq \delta $, then $\gab $ is a frame
  for $\lrd $. 
\end{coro}

\begin{proof}
We argue by contradiction and assume that $\gab $ is not a frame. By
condition (vi) of Theorem~\ref{th_char_frame} there exists a weak
limit $\Gamma \in W(\Lambda )$ and non-zero $f\in M^\infty (\rd )$, such that 
$V_gf \big| _\Gamma = 0$. Since $\rho (\Lambda ) \leq \delta $, we also
have $\rho (\Gamma )\leq  \delta $. By normalizing, we may
assume that $\norm{f}_{M^\infty } = \norm{V_g f}_\infty = 1$. For
$0< \epsilon < 1 - \omega _\delta (g)_{M^1}$ we find $z\in \rdd $
such that $|V_gf(z)| = |\langle f, \pi (z)g\rangle | > 1-\epsilon
$. By Lemma \ref{lemma_compactness}, $\Gamma$ is relatively separated and, in particular,
closed. Since $\rho (\Gamma ) \leq \delta $, there is a $\gamma \in \Gamma
$ such that $|z-\gamma | \leq \delta $. 
Consequently, since $V_gf \big| _\Gamma = 0$, we find that 
\begin{align*}
  1-\epsilon &< |\langle f, \pi (z)g\rangle - \langle f, \pi (\gamma
  )g\rangle | = |\langle f, \pi (z)g-  \pi (\gamma
  )g\rangle | \\
& = |\langle V_g f, V_g( \pi (z)g-  \pi (\gamma
  )g)\rangle | \\
&\leq \|V_gf \|_\infty \, \|V_g( \pi (z)g-  \pi (\gamma
  )g)\|_1 \\
&= \|f\|_{M^\infty } \, \|\pi (z)g-  \pi (\gamma
  )g\|_{M^1} \\
&\leq \omega _\delta (g)_{M^1} \, .
\end{align*}
Since we have chosen $1-\epsilon > \omega _\delta (g)_{M^1} $, we have
arrived at a contradiction. Thus $\gab $ is a frame. 
 \end{proof}

This theorem is analogous to Beurling's famous  sampling
theorem for  multivariate bandlimited functions~\cite{beurling66}. 
  The proof is in the style of ~\cite{OU12}.  

\subsection{Characterization of Riesz sequences} 
We now derive analogous results for Riesz sequences. 
\begin{theo}
\label{th_char_riesz}
Assume that $g\in M^1(\Rdst)$ and that $\Lambda \subseteq \Rtdst$ is
separated. Then the following are equivalent.
\begin{itemize}
\item[(i)] $\gab$ is a Riesz sequence in  $\ltrd$.

\item[(ii)] $\gab$ is a $p$-Riesz sequence in  $M^p(\rd )$ for some 
  $p\in [1,\infty ]$ (for all  $p\in [1,\infty ]$).

\item[(iii)] $\gab$ is an $\infty$-Riesz sequence in  $M^\infty(\rd )$,
  i.e., $C^* _{g,\Lambda} : \ell ^\infty (\Lambda) \to
  M^\infty(\Rdst)$ is bounded below. 
\item[(iv)] $C_{g,\Lambda} : M^1 \to \ell^1(\Lambda)$ is surjective.

\item[(v)] $C^* _{g,\Lambdastar } : \ell ^\infty (\Lambdastar) \to M^\infty(\Rdst)$ is bounded below for every
weak limit $ \Lambdastar \in W(\Lambda )$. 
\item[(vi)] $C^* _{g,\Lambdastar} : \ell ^\infty (\Lambdastar) \to M^\infty(\Rdst)$ is one-to-one for every
weak limit $ \Lambdastar \in W(\Lambda )$. 
\end{itemize}
\end{theo}
\begin{rem}{\rm (i)
Note that we are assuming that $\Lambda$ is separated. This is necessarily the case
if $\gab$ is a Riesz sequence (Lemma \ref{lemma_set_must_be}), but
needs to be assumed in some of  the other conditions.

(ii) For bandlimited functions Beurling \cite[Problem 3, p. 359]{be89} asked whether a characterization analogous to
(iii) $\Leftrightarrow$ (vi) in Theorem~\ref{th_char_riesz} holds for interpolating sequences. 
Note that in the context of bandlimited functions, the properties corresponding to (i) and (ii) are not equivalent.
}
\end{rem}
Before proving Theorem \ref{th_char_riesz}, we prove the following
continuity property of $C_{g,\Lambda }^*$ with respect to $\Lambda$. 
\begin{lemma}
\label{lemma_zero_seq}
Let $g\in M^1(\Rdst), g\neq 0$,  and let $\sett{\Lambda_n: n \geq 1}$ be
a sequence of uniformly separated subsets of $\Rtdst$,  i.e.,
\begin{equation} \label{unifsep}
\inf_n \sep(\Lambda_n) =\delta > 0\, .
\end{equation}
For every  $n \in \Nst$, let $c^n \in \ell^\infty(\Lambda_n)$ be such that $\norm{c^n}_\infty=1$ and
suppose that
\begin{align*}
\sum_{\lambda \in \Lambda_n} c^n_\lambda \pi(\lambda) g \longrightarrow 0
\mbox{ in } M^\infty(\Rdst), \qquad \mbox{as } n \longrightarrow \infty.
\end{align*}
Then there exist a subsequence $(n_k ) \subset \bN $,
points $\lambda_{n_k} \in \Lambda_{n_k}$, a separated set $\Gamma
\subseteq \Rtdst$,  and a
non-zero sequence $c \in \ell^\infty(\Gamma)$ such that
\begin{align*}
\Lambda_{n_k}-\lambda_{n_k} & \weakconv \Gamma,
\qquad \mbox{as } k \longrightarrow \infty \\
\text{and} \qquad \quad  
\sum_{\lambda \in \Gamma} c_\lambda \pi(\lambda) g &= 0.
\end{align*}
\end{lemma}
\begin{proof}
Combining the hypothesis~\eqref{unifsep} and
observation~\eqref{eq_sep_relsep},   we also  have the  uniform  relative separation
\begin{align}
\label{eq_unif_relsep}
\sup_n \rel(\Lambda_n) < \infty.
\end{align}
Since $\norm{c^n}_\infty=1$ for every  $n \geq 1$, we may  choose $\lambda_n \in \Lambda_n$ be such that
$\abs{c^n_{\lambda_n}} \geq 1/2$. Let $\theta_{\lambda,n}\in \bC $  such that
\begin{align*}
\theta_{\lambda,n} \pi(\lambda-\lambda_n)=\pi(-\lambda_n) \pi(\lambda),
\end{align*}
and consider the measures
$\mu_n := \sum_{\lambda \in \Lambda_n}
\theta_{\lambda,n} c^n_\lambda
\delta_{\lambda-\lambda_n}$. Then by Lemma \ref{lemma_norms_mes},
$\norm{\mu_n}_{\wmes} 
\lesssim \rel(\Lambda_n-\lambda_n) \norm{c^n}_\infty 
= \rel(\Lambda_n) \norm{c^n}_\infty \lesssim 1$. Using \eqref{eq_unif_relsep}
and Lemma \ref{lemma_compactness}, we may  pass  to a subsequence
such that (i) $\Lambda_n - \lambda_n \weakconv \Gamma$  for some relatively separated
set $\Gamma \subseteq \Rtdst$ and (ii)  $\mu_n \longrightarrow \mu$
in $\sigma(\wmes, \wtest)(\Rtdst)$  for some measure $\mu \in \wmes(\Rtdst)$. The uniform separation condition in
\eqref{unifsep} implies that $\Gamma$ is also separated.

Since $\supp(\mu_n) \subseteq \Lambda_n-\lambda_n$ it follows 
from Lemma \ref{lemma_mes_supp_a} that $\supp(\mu) \subseteq
\overline{\Gamma } = \Gamma $. Hence,
\begin{align*}
\mu = \sum_{\lambda \in \Gamma} c_\lambda \delta_\lambda,
\end{align*}
for some sequence of complex numbers $c$, and, 
by Lemma \ref{lemma_norms_mes}, $\norm{c}_\infty \leq \norm{\mu}_{\wmes} < \infty$.

From \eqref{unifsep} it follows that
for all $n \in \Nst$, $B_{\delta /2} (\lambda_n) \cap \Lambda_n = \sett{\lambda_n}$.
Let $\varphi \in C(\Rtdst)$ be real-valued, supported on
$B_{\delta / 2} (0)$ and such that $\varphi(0)=1$. Then
\begin{align*}
\abs{\int_{\Rtdst} \varphi \, d\mu}=\lim_n \abs{\int_{\Rtdst} \varphi \, d\mu_n}
=\lim_n \abs{c^n_{\lambda_n}} \geq 1/2.
\end{align*}
Hence $\mu \not = 0$ and therefore $c \not= 0$.

Finally, we  show that the short-time Fourier transform of
$\sum_\lambda c_\lambda \pi(\lambda) g $ is zero.
Let $z \in \Rtdst$ be arbitrary and  recall that by Lemma
\ref{lemma_stft} $V_{g} \pi(z)g \in \wtest(\Rtdst)$. Now we 
estimate
\begin{align*}
&\abs{\ip{\sum_{\lambda \in \Gamma} c_\lambda \pi(\lambda)g}{\pi(z)g}}
=
\abs{\sum_{\lambda \in \Gamma} c_\lambda \overline{V_{g} \pi(z)g}(\lambda)}
\\
&\quad=\abs{\int_{\Rtdst} \overline{V_{g}\pi(z)g} \,d\mu}
=\lim_n \abs{\int_{\Rtdst} \overline{V_{g}\pi(z)g} \,d{\mu_n}}
\\
&\quad=
\lim_n \abs{\ip{\sum_{\lambda \in \Lambda_n} \theta_{\lambda,n} c^n_\lambda \pi(\lambda-\lambda_n)g}{\pi(z)g}}
\\
&\quad\leq
\lim_n
\bignorm{\sum_{\lambda \in \Lambda_n} \theta_{\lambda,n} c^n_\lambda \pi(\lambda-\lambda_n)g}_{M^\infty}
\bignorm{g}_{M^1}
\\
&\quad=
\lim_n
\bignorm{\pi(-\lambda_n) \sum_{\lambda \in \Lambda_n} c^n_\lambda \pi(\lambda)g}_{M^\infty} \bignorm{g}_{M^1}
\\
&\quad=
\lim_n
\bignorm{\sum_{\lambda \in \Lambda_n} c^n_\lambda \pi(\lambda)g}_{M^\infty} \bignorm{g}_{M^1}
=0.
\end{align*}
We have shown that
$V_g (\sum_{\lambda \in \Gamma} c_\lambda \pi(\lambda)g) \equiv 0$ and
thus  $\sum_{\lambda \in \Gamma} c_\lambda \pi(\lambda)g
\equiv 0$,  as desired.
\end{proof}

\begin{proof}[Proof of Theorem \ref{th_char_riesz}]
The equivalence of  (i), (ii),  and (iii) follows from Corollary
\ref{coro_main_mol}(b), and  the equivalence
of (iii) and (iv) follows by duality.

{\bf (iv) $\Rightarrow$ (v)}.
Assume (iv) and consider a sequence $\Lambda - z_n \weakconv \Lambdastar$. Let
$\lambda \in \Lambdastar$ be
arbitrary and  let $\sett{\lambda_n: n \in \Nst} \subseteq \Lambda$ be a sequence such that
$\lambda_n -z_n \longrightarrow \lambda$. By the open map theorem, every sequence $c \in \ell^1(\Lambda)$
with $\norm{c}_1=1$ has a preimage $c=C_{g,\Lambda}(f)$ with $\norm{f}_{M^1} \lesssim 1$.
With the covariance property~\eqref{eq_cov_c}  we deduce that there exist $f_n \in M^1(\rd )$, such
that $c=C_{g,\Lambda -z_n}(f_n)$ and $\|f_n\|_{M^1} \lesssim 1$.

In particular, for each $n \in \Nst$ there exists an interpolating function $h_n \in M^1(\Rdst)$ such that $\norm{V_g h_n}_1
\lesssim 1$,
$V_g h_n(\lambda_n-z_n)=1$ and $V_g h_n \equiv 0$ on $\Lambda-z_n \setminus \sett{\lambda_n-z_n}$.
By passing to a subsequence we may assume that $h_n \longrightarrow h$ in $\sigma(M^1,\modzero)$. It follows
that $\norm{h}_{M^1} \lesssim 1$. Since  $V_g h_n \longrightarrow V_g h$ uniformly
on compact sets  by Lemma \ref{lemma_stft}, we obtain that
\begin{align*}
V_g h(\lambda) = \lim_n V_g h_n(\lambda_n-z_n)=1.
\end{align*}
Similarly, given $\gamma \in \Lambdastar \setminus \sett{\lambda}$, there exists a sequence
$\sett{\gamma_n: n \in \Nst} \subseteq \Lambda$ such that
$\gamma_n -z_n \longrightarrow \gamma$. Since $\lambda \not= \gamma$, for $n\gg 0$
we have that $\gamma_n \not= \lambda_n$ and consequently $V_g h_n(\gamma_n-z_n)=0$. It follows
that $V_g h(\gamma)=0$.

Hence, we have shown that for each $\lambda \in \Lambdastar$ there
exists an interpolating  function $h_\lambda \in M^1(\Rdst)$
such that $\norm{h_\lambda}_{M^1} \lesssim 1$, $V_g h_\lambda(\lambda)=1$ and
$V_g h_\lambda\equiv 0$ on $\Lambdastar \setminus \sett{\lambda}$.
Given an arbitrary sequence $c \in \ell^1(\Lambdastar)$ we consider
\begin{align*}
f := \sum_{\lambda \in \Lambdastar} c_\lambda h_\lambda.
\end{align*}
It follows that $f \in M^1(\Rdst)$ and that $C_{g,\Lambdastar} f = c$. Hence, $C_{g,\Lambdastar}$ 
is onto $\ell ^1(\Gamma )$,  and therefore $C^*_{g,\Lambdastar}$ is bounded below.

{\bf (v) $\Rightarrow$ (vi)} is clear.

{\bf (vi) $\Rightarrow$ (iii)}.
Suppose that (iii) does not hold. Then there exists a sequence $\sett{c^n: n \in \Nst}
\subseteq \ell^\infty(\Lambda)$ such that $\norm{c^n}_\infty = 1$ and
\begin{align*}
\bignorm{\sum_{\lambda \in \Lambda} c^n_\lambda \pi(\lambda)g}_{M^\infty} \longrightarrow 0,
\mbox{ as }n \longrightarrow \infty.
\end{align*}
We now apply Lemma \ref{lemma_zero_seq} with $\Lambda_n := \Lambda$ and obtain a set $\Gamma \in W(\Lambda)$
and a non-zero sequence $c \in \ell^\infty(\Gamma)$ such that $\sum
_{\lambda \in \Gamma } c_\lambda \pi (\lambda )g = C^* _{g,\Gamma}(c)=0$. This contradicts (vi).
\end{proof}

\section{Deformation of sets and Lipschitz convergence}
\label{sec_regconv}

The characterizations  of Theorem~\ref{th_char_frame} suggest that Gabor frames are
invariant under ``weak deformations'' of $\Lambda $. One might expect
that if $\cG (g,\Lambda )$ is a frame  and $\Lambda '$ is  close to
$\Lambda $ in the weak sense, then $\cG (g, \Lambda ')$ is also a
frame. This view is too simplistic. Just choose $\Lambda _n = \Lambda \cap
B_n(0)$, then $\Lambda _n \weakconv \Lambda $, but $\Lambda _n$ is a
finite set and thus   $\cG (g,\Lambda _n)$ is
never a frame. For a  deformation result we need to introduce  a
finer notion of convergence.

Let $\Lambda \subseteq \Rdst$ be a (countable)  set. We consider a sequence 
$\sett{\Lambda_n: n \geq 1}$ of subsets of $\Rdst$
produced in the following way. For each $n \geq 1$, let
$\map_n: \Lambda \to \Rdst$ be a map and let
$\Lambda_n := \map_n(\Lambda) = \sett{\mapn{\lambda}: \lambda \in \Lambda}$. 
We assume that $\mapn{\lambda} \longrightarrow \lambda$, as $n \longrightarrow \infty$,
for all $\lambda \in \Lambda$. The sequence of sets $\sett{\Lambda_n: n \geq 1}$ together with the maps
$\sett{\map_n: n \geq 1}$ is called a \emph{deformation} of $\Lambda$. We think of each sequence of points
$\sett{\mapn{\lambda}: n \geq 1}$ as a (discrete) path moving towards
the endpoint  $\lambda$.

 We will often  say that $\sett{\Lambda_n: n \geq 1}$ is a deformation
 of $\Lambda$, with the 
understanding that a sequence of underlying maps $\sett{\map_n: n \geq  1}$  is also given. 

\begin{definition}
A deformation $\sett{\Lambda_n: n \geq 1}$ of $\Lambda $ is called
\emph{Lipschitz}, denoted by  $\Lambda_n \lipconv \Lambda$, 
if  the following two  conditions hold:

{\bf \regone } Given $R>0$,
\begin{align*}
\sup_{
\stackrel{\lambda, \lambda' \in \Lambda}{\abs{\lambda-\lambda'} \leq R}}
\abs{(\mapn{\lambda} - \mapn{\lambda'}) - (\lambda - \lambda')} \rightarrow 0,
\quad \mbox {as } n \longrightarrow \infty.
\end{align*}

{\bf \regtwo }  Given $R>0$, there exist $R'>0$ and $n_0 \in \Nst$
such that if 
$\abs{\mapn{\lambda} - \mapn{\lambda'}} \leq R$ for \emph{some} $n
\geq n_0$ and some $\lambda, \lambda' \in \Lambda$,  then
$\abs{\lambda-\lambda'} \leq R'$.

\end{definition}
 Condition \regone\ means that $\mapn{\lambda} - \mapn{\lambda'}
 \longrightarrow \lambda - \lambda'$ uniformly in
 $\abs{\lambda-\lambda'}$. In particular, by fixing $\lambda '$, we
 see that Lipschitz convergence implies the weak convergence $\Lambda _n
\weakconv \Lambda $. Furthermore,  if $\{\Lambda_n: n \geq 1\}$ is
Lipschitz convergent to $\Lambda$, then 
so is every subsequence $\sett{\Lambda_{n_k}: k \geq 1}$.

\begin{example} Jitter error: {\rm 
Let $\Lambda \subseteq \Rdst$ be relatively separated and let $\sett{\Lambda_n: n \geq 1}$ be a deformation
of $\Lambda$. If $\sup_\lambda \abs{\mapn{\lambda} - \lambda} \longrightarrow 0$,
as $n \longrightarrow \infty$, then $\Lambda_n \lipconv
\Lambda$. }
\end{example}

\begin{example} Linear deformations: {\rm 
Let $\Lambda = A \zdd \subseteq \rdd$, with $A$ an invertible $2d\times 2d$ matrix, $\Lambda _n = A_n \zdd $
for a sequence of invertible $2d\times 2d$-matrices and assume that $\lim
A_n = A$. Then $\Lambda_n \lipconv \Lambda $. In
this case conditions \regone\  and \regtwo\ are easily checked by taking
$\tau_n = A_n A\inv$.  }
\end{example}

The third class of examples contains  differentiable, nonlinear deformations. 

\begin{lemma}
\label{lemma_reg_conv}
Let $p \in (d, \infty]$. For each $n \in \Nst $, let $T_n=(T_n^1, \ldots, T_n^d): \Rdst \to \Rdst$ be a 
map such that each coordinate  function $T^k_n:\Rdst \to \Rst$ is continuous,
locally integrable and has a weak derivative in $L^p_{loc}(\Rdst)$. 
Assume that
\begin{align*}
&T_n(0)= 0,
\\
&\abs{DT_n-I} \longrightarrow 0 \,\,\,   \mbox{ in } L^p(\Rdst).
\end{align*}
(Here, $DT_n$ is the Jacobian matrix consisting  of the  partial
derivatives of $T_n$ and the second condition means 
that each entry of the matrix $DT_n-I$ tends to $0$ in $L^p$.)

Let $\Lambda \subseteq \Rdst$ be a relatively separated set and consider the deformation
$\Lambda_n := T_n(\Lambda)$ (i.e $\map_n := {T_n}\big| _{\Lambda})$.
Then $\Lambda_n $ is Lipschitz convergent to $ \Lambda$.
\end{lemma}
\begin{rem}{\rm 
In particular,  the hypothesis of Lemma \ref{lemma_reg_conv} is satisfied by every sequence of differentiable
maps $T_n: \Rdst \to \Rdst$ such that}
\begin{align*}
&T_n(0)= 0,
\\
&\sup_{z\in\Rtdst} \abs{DT_n(z)-I} \longrightarrow 0.
\end{align*}
\end{rem}
\begin{proof}[Proof of Lemma \ref{lemma_reg_conv}]
Let $\alpha := 1 - \tfrac{d}{p} \in (0,1]$. We use the following Sobolev embedding known as
Morrey's inequality (see for example \cite[Chapter 4, Theorem
3]{evga92}). If $f:
\Rdst \to \Rst$  is locally integrable and possesses  a weak
derivative in $L^p(\Rdst)$, then $f$ is $\alpha $-H\"older continuous
(after being redefined on a  set of  measure zero). If $x,y \in \Rdst$, then 
\begin{align*}
\abs{f(x)-f(y)} \lesssim \norm{\nabla f}_{L^p(\Rdst)}  \abs{x-y}^\alpha,
\qquad x,y \in \Rdst.
\end{align*}
Applying Morrey's inequality to each coordinate function of  $T_n-I$
we obtain that there is a constant $C>0$ such that,
for $x,y \in \Rdst$,
\begin{align*}
\abs{(T_n x - T_n y) - (x - y)} = 
\abs{(T_n - I)x - (T_n - I)y} \leq  
C \norm{DT_n - I}_{L^p(\Rdst)}\,  \abs{x-y}^\alpha.
\end{align*}
Let  $\epsilon _n = C\norm{DT_n - I}_{L^p(\Rdst)}$,
where $\norm{DT_n - I}_{L^p(\Rdst)}$ is the $L^p$-norm of
$\abs{DT_n(\cdot) - I}$. Then $\epsilon _n \to 0$ by assumption and 
\begin{align}
\label{eq_lemma_p}
\abs{(T_n x - T_n y) - (x - y)} \leq \varepsilon_n \abs{x-y}^\alpha,
\qquad x,y \in \Rdst.
\end{align}
Choose  $x=\lambda$ and $y=0$, then 
$T_n(\lambda) \longrightarrow \lambda$ for all $\lambda \in
\Lambda$ (since $T_n(0)=0$). Hence
$\Lambda_n$ is a deformation of $\Lambda$. 

If $\lambda, \lambda' \in \Lambda$  and 
$\abs{\lambda-\lambda'} \leq R$, then  \eqref{eq_lemma_p} implies  that
\begin{align*}
\abs{(T_n \lambda - T_n \lambda') - (\lambda - \lambda')}
\leq \varepsilon_n R^\alpha.
\end{align*}
Thus condition \regone\  is satisfied. 

For condition \regtwo , choose  $n_0$  such that 
$\varepsilon_n \leq 1/2$ for $n \geq n_0$. If $|\lambda - \lambda '|
\leq 1$, there is nothing to show (choose $R' \geq 1$).
If $|\lambda -\lambda '| \geq 1$ and  $\abs{T_n \lambda -T_n\lambda'} \leq R$ for some $n \geq n_0$,
$\lambda, \lambda' \in \Lambda$, then by \eqref{eq_lemma_p}  we obtain
\begin{align*}
\abs{(T_n \lambda - T_n \lambda') - (\lambda - \lambda')} \leq 1/2 \abs{\lambda-\lambda'}^\alpha
\leq 1/2 \abs{\lambda-\lambda'}.
\end{align*}
This implies that 
\begin{align}
\label{eq_super_bound}
\abs{\lambda-\lambda'} \leq 2 \abs{(T_n \lambda - T_n \lambda')},
\mbox{ for all }n\geq n_0.
\end{align}
Since $\abs{T_n \lambda -T_n\lambda'} \leq R$, we conclude that
$\abs{\lambda -\lambda'} \leq 2R $, and we may actually choose $R'=
\max (1,2R) $ in condition \regtwo. 
\end{proof}

\begin{rem}
{\rm Property \regone\ can be proved under slightly weaker conditions on
  the convergence of $DT_n - I$. In fact,  we
  need~\eqref{eq_lemma_p} to hold only for $|x-y|\leq R$. Thus it
  suffices to assume  locally uniform convergence in $L^p$, i.e.,
$$
\sup _{y\in \rdd } \int _{B_R(y)} |DT_n (x) - I|^p \, dx  \to 0
$$
 for   all $R>0$. }  
\end{rem}
We prove some technical lemmas about Lipschitz convergence.
\begin{lemma}
\label{lemma_dens}
Let $\sett{\Lambda_n: n \geq 1}$ be a deformation of a  relatively
separated set $\Lambda \subseteq \Rdst$.  Then the following hold.
\begin{itemize}
\item[(a)] If $\Lambda_n$ is Lipschitz convergent  to $\Lambda$ and
$\sep(\Lambda)>0$, then $\liminf_n \sep(\Lambda _n) >0$.
\item[(b)] If $\Lambda_n$ is  Lipschitz convergent to $\Lambda$, then $\limsup_n \rel(\Lambda_n) < \infty$.
\item[(c)] If $\Lambda_n$ is Lipschitz convergent to $\Lambda$ and
$\gap(\Lambda) <\infty$, then $\limsup_n \gap(\Lambda_n)<\infty$.
\end{itemize}
\end{lemma}
\begin{proof}
(a) By assumption  $\delta := \sep(\Lambda)>0$. Using \regtwo, let $n_0 \in \Nst$ and
$R'>0$ be such that if $\abs{\mapn{\lambda} - \mapn{\lambda'}} \leq
\delta/2$ for some $\lambda,\lambda' \in \Lambda$ and $n \geq n_0$, 
then $\abs{\lambda-\lambda'} \leq R'$. By \regone ,  choose  $n_1 \geq
n_0$  such that for  $n \geq n_1$ 
\begin{align*}
\sup_{\abs{\lambda-\lambda'} \leq R'} \abs{(\mapn{\lambda} - \mapn{\lambda'}) - (\lambda - \lambda')} < \delta/2.
\end{align*}
\textbf{Claim.}  $\sep(\Lambda_n) \geq \delta/2$ for $n\geq n_1$. 

If the claim is not true, then for some $n \geq n_0$ there exist two distinct 
points $\lambda, \lambda' \in \Lambda$ such that 
$\abs{\mapn{\lambda}-\mapn{\lambda'}} \leq \delta/2$.  Then $\abs{\lambda-\lambda'} \leq R'$ and consequently
\begin{align*}
\abs{\lambda-\lambda'} \leq \abs{(\mapn{\lambda} - \mapn{\lambda'}) - (\lambda - \lambda')} +
\abs{\mapn{\lambda}-\mapn{\lambda'}}
< \delta,
\end{align*}
contradicting the fact that $\Lambda$ is $\delta$-separated.

(b) Since $\Lambda$ is relatively separated we can split it into
finitely many separated sets $\Lambda=\Lambda^1 \cup
\ldots \cup \Lambda^L$  with $\sep (\Lambda ^k) >0$. 

Consider the sets defined by restricting the deformation $\map_n$ to each $\Lambda^k$
\begin{align*}
\Lambda^k_n := \sett{\mapn{\lambda}: \lambda \in \Lambda^k}. 
\end{align*}

As proved above in (a), there exists $n_0 \in \Nst$ and $\delta>0$
such that $\sep(\Lambda^k_n) \geq \delta$ for all $n \geq n_0$
and $1 \leq k \leq L$.  Therefore, using \eqref{eq_sep_relsep},
\begin{align*}
\rel(\Lambda_n) \leq \sum_{k=1}^L \rel(\Lambda^k_n) \lesssim L \delta^{-d},
\qquad n \geq n_0,
\end{align*}
and the conclusion follows.

(c) By (b) we may assume that each $\Lambda_n$ is relatively separated.
Assume that $\gap(\Lambda) <\infty$. Then there exists $r>0$ such that every cube
$Q_r(z) := z + [-r,r]^{d}$ intersects $\Lambda$. By  \regone ,  there
is  $n_0 \in \Nst$ such that for $n \geq n_0$, 
\begin{align}
\label{eq_lemma_reg}
\sup_{\stackrel{\lambda, \lambda' \in \Lambda}{\abs{\lambda-\lambda'}_\infty \leq 6r}}
\abs{(\mapn{\lambda} - \mapn{\lambda'}) - (\lambda - \lambda')}_\infty \leq r.
\end{align}
Let $R:=8r$ and $n \geq n_0$. We will  show that every cube $Q_R(z)$
intersects $\Lambda_n$. 
This will give a uniform upper bound for $\gap(\Lambda_n)$.
Suppose on the contrary that some cube $Q_R(z)$ does not meet $\Lambda_n$ and consider a larger radius
$R' \geq R$ such that $\Lambda_n$ intersects the boundary but not the interior of $Q_{R'}(z)$.
(This is possible because $\Lambda_n$ is relatively separated and therefore closed.) Hence, there exists
$\lambda \in \Lambda$ such that
$\abs{\mapn{\lambda}-z}_\infty=R'$. Let us write 
\begin{align*}
&(z-\mapn{\lambda})_k = \delta_k c_k, \qquad k=1,\ldots,d,
\\
&\delta_k \in \{-1,1\}, \qquad k=1,\ldots,d,
\\
&0 \leq c_k \leq R', \qquad k=1,\ldots,d \, ,
\end{align*}
and $c_k = R' $ for some $k$.  We now argue that we can select a point $\gamma \in \Lambda$ such that
\begin{align}
\label{eq_lga}
(\lambda-\gamma)_k = -\delta_k c'_k, \qquad k=1,\ldots,d,
\end{align}
with coordinates 
\begin{align}
\label{eq_lgb}
2r \leq c'_k \leq 6r, \qquad k=1,\ldots,d.
\end{align}
Using the fact that $\Lambda$ intersects each of the cubes
$\sett{Q_r(2rj):j\in \Zdst}$, 
we first select an index $j \in \Zdst$ such that $\lambda \in
Q_r(2rj)$. Second, we define a new index $j' \in \Zdst$ by $j_k' = j_k
+2\delta _k$ for $k=1, \dots , d$. 
We finally select a point $\gamma \in \Lambda \cap Q_r(2rj')$.
This procedure guarantees that \eqref{eq_lga} and \eqref{eq_lgb} hold
true. See Figure~\ref{fig:1}. 
\begin{figure}
\centering
\includegraphics[width=0.45\textwidth,natwidth=356,natheight=270]{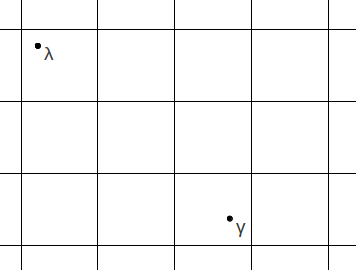}
\caption{The selection of the point $\gamma$ satisfying \eqref{eq_lga} and \eqref{eq_lgb}.}
\label{fig:1}
\end{figure}

Since by \eqref{eq_lga} and \eqref{eq_lgb} $\abs{\lambda-\gamma}_\infty \leq 6r$, we can use \eqref{eq_lemma_reg} to
obtain
\begin{align*}
(\mapn{\lambda}-\mapn{\gamma})_k = -\delta_k c''_k, \qquad k=1,\ldots,d,
\end{align*}
with coordinates
\begin{align*}
r \leq c''_k \leq 7r , \qquad k=1,\ldots,d.
\end{align*}
We write $(z-\mapn{\gamma})_k=(z-\mapn{\lambda})_k+(\mapn{\lambda}-\mapn{\gamma})_k=\delta_k(c_k-c''_k)$ and note
that $-7r \leq c_k-c''_k \leq R' - r$. Hence,
\begin{align*}
\abs{z-\mapn{\gamma}}_\infty \leq \max\{R'-r,7r\}=R'-r,
\end{align*}
since $7r = R-r \leq R'-r$. This shows that $Q_{R'-r}(z)$ intersects $\Lambda_n$, contradicting the choice of $R'$.
\end{proof}
The following lemma relates Lipschitz convergence to 
the weak-limit techniques.
\begin{lemma}
\label{lemma_reg_key}
Let $\Lambda \subseteq \Rdst$ be relatively separated  and let
$\sett{\Lambda_n: n \geq 1}$ be a  Lipschitz 
deformation of $\Lambda$. Then the following holds.

\begin{itemize}
\item[(a)] Let $\Gamma \subseteq \Rdst$ and $\sett{\lambda_n: n \geq
    1} \subseteq \Lambda$ be some sequence in $\Lambda $. If 
$\Lambda_n - \mapn{\lambda_n} \weakconv \Gamma$,  then $\Gamma \in
W(\Lambda)$. 

\item[(b)] Suppose that $\Lambda$ is relatively dense and
$\sett{z_n: n \geq 1} \subseteq \Rdst$ is an arbitrary sequence.
If $\Lambda_n - z_n \weakconv \Gamma$,  then $\Gamma \in W(\Lambda)$. 
\end{itemize}

\end{lemma}
\begin{proof}
 (a) We first note that $\Gamma$ is relatively separated. Indeed,
 Lemma \ref{lemma_dens} says that
\begin{align*}
\limsup_{n \rightarrow \infty} \rel(\Lambda_n - \mapn{\lambda_n})
= \limsup_{n \rightarrow \infty} \rel(\Lambda_n) < \infty,
\end{align*}
and Lemma~\ref{lemma_compactness}(c)  implies that $\Gamma$ is relatively
 separated (and in particular closed). 

By extracting  a  subsequence,  we may assume that $\Lambda - \lambda_n \weakconv \Gamma'$ for
some relatively separated set $\Gamma' \in W(\Lambda )$. We will show
that $\Gamma'=\Gamma$ and consequently $\Gamma \in W(\Lambda)$.

Let $R>0$ and $0<\varepsilon \leq 1$ be given. By \regone, there exists $n_0 \in \Nst$ such that
\begin{align}
\label{eq_lam_lamp}
\lambda,\lambda' \in \Lambda , \abs{\lambda-\lambda'}\leq R, n \geq
n_0  \Longrightarrow  \abs{(\mapn{\lambda} - \mapn{\lambda'}) -
  (\lambda - \lambda')} \leq \varepsilon \, .
\end{align}

If   $n \geq n_0$  and $z \in (\Lambda - \lambda_n) \cap
B_R(0)$,  then there exists $\lambda \in \Lambda$ such that $z=\lambda-\lambda_n$ and $\abs{\lambda-\lambda_n} \leq R$.
Consequently~\eqref{eq_lam_lamp} implies that
\begin{align*}
\abs{(\mapn{\lambda} - \mapn{\lambda_n}) - z} =
\abs{(\mapn{\lambda} - \mapn{\lambda_n}) - (\lambda-\lambda_n)} \leq \varepsilon.
\end{align*}
This shows that
\begin{align}
\label{eq_subsets}
(\Lambda - \lambda_n) \cap B_R(0) \subseteq (\Lambda_n -
\mapn{\lambda_n})+B_\varepsilon(0) \,\, 
\mbox{ for  } n \geq n_0\,  .
\end{align}
Since $\Lambda - \lambda_n \weakconv \Gamma'$ and $\Lambda_n -
\mapn{\lambda_n} \weakconv \Gamma$, it follows from 
\eqref{eq_subsets} and Lemma \ref{lemma_weak_inc} that
$\Gamma' \subseteq \overline{\Gamma}=\Gamma$.

For the reverse  inclusion, let again $R>0$ and $0< \varepsilon \leq 1$. Let $R'>0$ and $n_0 \in \Nst$ be the
numbers associated with $R$ in \regtwo. Using \regone, choose  $n_1
\geq n_0$  such that
\begin{align}
\label{eq_lam_lamp_2} 
\lambda,\lambda' \in \Lambda, \abs{\lambda-\lambda'}\leq R', n \geq
n_1 \quad \Longrightarrow \,\, 
\abs{(\mapn{\lambda} - \mapn{\lambda'}) - (\lambda - \lambda')} \leq
\varepsilon \, .
\end{align}

If  $n \geq n_1$ and  $z \in (\Lambda_n - \mapn{\lambda_n}) \cap
B_R(0)$,  then 
$z=\mapn{\lambda}-\mapn{\lambda_n}$  for some $\lambda \in \Lambda$ and $\abs{\mapn{\lambda}-\mapn{\lambda_n}} \leq R$.
Condition \regtwo\, now  implies that
$\abs{\lambda-\lambda_n} \leq R'$ and therefore, using
\eqref{eq_lam_lamp_2} with $\lambda'=\lambda_n$, we get
\begin{align*}
\abs{z- (\lambda - \lambda_n)} =
\abs{(\mapn{\lambda} - \mapn{\lambda_n}) - (\lambda-\lambda_n)} \leq \varepsilon.
\end{align*}
Hence  we have proved that
\begin{align*}
(\Lambda_n - \mapn{\lambda_n}) \cap B_R(0) \subseteq (\Lambda - \lambda_n)+B_\varepsilon(0),
\mbox{ for   } n\geq n_1\,  .
\end{align*}
Since $\Lambda_n - \mapn{\lambda_n} \weakconv \Gamma$ and $\Lambda -
\lambda_n \weakconv \Gamma'$,  Lemma \ref{lemma_weak_inc} implies  that $\Gamma \subseteq
\overline{\Gamma'}=\Gamma'$. In conclusion $\Gamma ' = \Gamma \in
W(\Lambda )$,  as desired.

(b)  Since $\gap(\Lambda) < \infty$, Lemma \ref{lemma_dens}(c)  implies that
$\limsup_n \gap(\Lambda_n) < \infty$. By omitting finitely many $n$,
there exists  $L>0$ such that $\Lambda_n + B_L(0) = \Rdst$  for all $n \in
\Nst$. This implies the existence of a sequence $\sett{\lambda_n: n \geq 1} \subseteq \Lambda$ such that
$\abs{z_n-\mapn{\lambda_n}} \leq L$. By passing to a subsequence we
may assume  that $z_n-\mapn{\lambda_n} 
\longrightarrow z_0$  for some  $z_0 \in \Rdst$. 

Since $\Lambda_n - z_n \weakconv \Gamma$ and $z_n-\mapn{\lambda_n} \longrightarrow z_0$,
it follows that $\Lambda_n -\mapn{\lambda_n} \weakconv \Gamma + z_0$. By (a), we deduce that
$\Gamma + z_0 \in W(\Lambda)$ and thus  $\Gamma \in W(\Lambda)$, as desired.
\end{proof}

\section{Deformation of Gabor systems}
We now  prove the main results  on the  deformation of Gabor
systems. The proofs combine the characterization of non-uniform Gabor
frames and Riesz sequences without inequalities and the fine details
of Lipschitz convergence.

First we formulate the stability of Gabor frames under a
class of nonlinear deformations. 

\begin{theo}
\label{th_per_frame}
Let $g \in M^1(\Rdst)$,  $\Lambda \subseteq \Rtdst$  and assume that
$\cG (g,\Lambda )$ is a frame for   $\lrd $. 

If $\Lambda _n$ is  Lipschitz convergent to $\Lambda $, then $\cG (g,
\Lambda _n)$ is a frame for all sufficiently large $n$. 
\end{theo}

 Theorem \ref{th_main_intro}(a)  of  the Introduction now follows by
 combining  Theorem~\ref{th_per_frame} and Lemma
 \ref{lemma_reg_conv}. Note that in Theorem~\ref{th_main_intro} we may
 assume without loss of generality that  $T_n(0)=0$, because the
 deformation problem is invariant under translations. 
\begin{proof}

  Suppose that $\mathcal{G}(g,\Lambda)$ is a frame. 
According to Lemma \ref{lemma_set_must_be},   $\Lambda$ is relatively separated and relatively dense.
Now suppose that the conclusion does not hold. By passing to a
subsequence we may assume that  $\mathcal{G}(g,\Lambda_n)$ fails to be
a frame for all $n\in \bN $. 

By Theorem \ref{th_char_frame} every $\cG (g, \Lambda _n)$ also fails
to be an $\infty $-frame for $M^\infty (\rd )$.  It follows that for every $n \in \Nst$ there exist $f_n \in
M^\infty(\Rdst)$ such that
$\norm{V_g f_n}_{\infty}=1$ and
\begin{align*}
\norm{C_{g, \Lambda_n}(f_n)}_{\ell^\infty(\Lambda_n)}=
\sup_{\lambda \in \Lambda} \abs{V_g f_n(\mapn{\lambda})} \longrightarrow 0,
\qquad \mbox{as } n \longrightarrow \infty.
\end{align*}
For each $n \in \Nst$, let $z_n \in \Rtdst$ be such that $\abs{V_g f_n(z_n)} \geq 1/2$ and let us consider
$h_n:=\pi(-z_n) f_n$. By passing to a subsequence we may assume that $h_n \rightarrow h$ in $\sigma(M^\infty,M^1)$
for some function $h \in M^\infty$. Since $\abs{V_g h_n(0)}=\abs{V_g
  f_n(z_n)} \geq 1/2$, it follows that $\abs{V_g h(0)} \geq 1/2$ and
the weak$^*$-limit  $h$ is not zero. 

 In addition, by Lemma \ref{lemma_dens}
\begin{align*}
\limsup_{n\rightarrow \infty} \rel(\Lambda_n - z_n) = \limsup_{n\rightarrow \infty} \rel(\Lambda_n) < \infty.
\end{align*}
Hence, using Lemma \ref{lemma_compactness} and passing to a further subsequence,
we may assume that $\Lambda_n - z_n \weakconv \Gamma$, for some relatively separated set $\Gamma \subseteq \Rtdst$.
Since $\Lambda$ is relatively dense,  Lemma \ref{lemma_reg_key}
guarantees that $\Gamma \in W(\Lambda)$.

Let $\gamma \in \Gamma$ be arbitrary.  Since $\Lambda_n - z_n \weakconv \Gamma$,
there exists a sequence $\sett{\lambda_n}_{n\in\Nst} \subseteq \Lambda$ such that $\mapn{\lambda_n} - z_n \rightarrow
\gamma$. By Lemma \ref{lemma_stft}, the fact that $h_n \rightarrow h$ in $\sigma(M^\infty,M^1)$ implies that $V_g h_n
\rightarrow V_g h$ uniformly on compact sets. Consequently, by \eqref{eq_tf_stft},
\begin{align*}
\abs{V_g h(\gamma)} = \lim_n \abs{V_g h_n(\mapn{\lambda_n} - z_n)} = \lim_n \abs{V_g f_n(\mapn{\lambda_n})}=0. 
\end{align*}
Hence, $h \not\equiv 0$ and $V_g h \equiv 0$ on $\Gamma \in W(\Lambda)$. According to Theorem
\ref{th_char_frame}(vi), $\mathcal{G}(g,\Lambda)$ is not a frame, thus
contradicting the initial assumption.  
\end{proof}

The corresponding deformation result for Gabor Riesz sequences reads
as follows. 

\begin{theo}
\label{th_per_riesz}
Let $g \in M^1(\Rdst)$,  $\Lambda \subseteq \Rtdst$,  and assume that
$\cG (g,\Lambda )$ is a Riesz sequence in  
 $\lrd $. 

If $\Lambda _n$ is  Lipschitz convergent to $\Lambda $, then $\cG (g,
\Lambda _n)$ is a Riesz sequence  for all sufficiently large $n$. 
\end{theo}

\begin{proof}
  Assume that $\mathcal{G}(g,\Lambda)$ is a Riesz sequence. Lemma~\ref{lemma_set_must_be}
 implies that $\Lambda$ is separated. With Lemma \ref{lemma_dens} we
 may  extract   a   subsequence such that each $\Lambda_n$ is separated with a uniform separation constant, i.e.,
\begin{align}
\label{eq_unif_sep}
\inf_{n \geq 1} \sep(\Lambda_n) >0.
\end{align}
We argue by contradiction and assume  that the conclusion  does not hold. By passing to
a further   subsequence,  we may assume that 
$\mathcal{G}(g,\Lambda_n)$ fails to be a Riesz sequence for all $n\in
\bN$. As a consequence of  Theorem~\ref{th_char_riesz}(iii), there exist sequences $c^n \in 
\ell^\infty(\Lambda_n)$ such that $\norm{c^n}_\infty=1$
and
\begin{align}
\label{eq_goestozero}
\norm{C^*_{g,\Lambda_n}(c^n)}_{M^\infty}=
\Bignorm{\sum_{\lambda \in \Lambda} c^n_{\mapn{\lambda}} \pi(\map_n{(\lambda}))g}_{M^\infty} \longrightarrow 0,
\mbox{ as }n \longrightarrow \infty.
\end{align}
Thus $g, \Lambda _n, c^n$ satisfy the assumptions of
Lemma~\ref{lemma_zero_seq}. The conclusion of
Lemma~\ref{lemma_zero_seq} yields a 
subsequence $(n_k)$, a separated set $\Gamma \subseteq \Rtdst$, a 
non-zero sequence $c \in \ell^\infty(\Gamma)$, and a sequence of points $\sett{\lambda_{n_k}: k \geq 1} \subseteq
\Lambda$
such that 
\begin{align*}
\Lambda_{n_k} - \tau _{n_k} (\lambda_{n_k}) \weakconv \Gamma
\end{align*}
and 
$$\sum _{\gamma  \in \Gamma} c_\gamma \pi (\gamma ) g =  C^*_{g,
  \Gamma} (c)=0 \, .
$$
By Lemma \ref{lemma_reg_key}, we conclude that $\Gamma \in
W(\Lambda)$.  According to condition (vi) of  Theorem \ref{th_char_riesz},
 $\mathcal{G}(g,\Lambda)$ is not  a Riesz sequence, which is a  contradiction.
\end{proof}
\begin{rem} Uniformity of the bounds: {\rm 
Under the assumptions of Theorem \ref{th_per_frame} it follows that there exists $n_0 \in \Nst$
and constants $A,B>0$ such that for $n \geq n_0$,
\begin{align*}
A \norm{f}_2^2 \leq \sum_{\lambda \in \Lambda} \abs{\ip{f}{\pi(\mapn{\lambda})g}}^2 
\leq B \norm{f}_2^2,
\qquad f \in L^2(\Rdst).
\end{align*}
The uniformity of the upper bound follows from the fact that $g \in M^1(\Rdst)$ and
$\sup_n \rel(\Lambda_n) < \infty$ (cf. Section \ref{sec_maps}).
For the lower bound, note that the proof
of Theorem \ref{th_per_frame} shows that there exists a constant
$A'>0$ such that, for $n \geq n_0$, 
\begin{align*}
A' \norm{f}_{M^\infty} \leq \sup_{\lambda \in \Lambda} 
\abs{\ip{f}{\pi(\mapn{\lambda})g}},
\qquad f \in M^\infty(\Rdst).
\end{align*}
 This property  implies a uniform $L^2$-bound as is made explicit in Remarks~\ref{rem_unif_loc}
and \ref{rem_unif_loc_2}.}
\end{rem}

To show why local preservation of differences is related to the
stability of Gabor frames, let us consider the following example.
\begin{example}
\label{ex_go_wrong}
{\rm From~\cite{asfeka13} or from Theorem~\ref{th_per_frame} we know
  that   if $g\in M^1(\rd )$ and  $\mathcal{G}(g,\Lambda)$ is a frame, then  $\cG
(g, (1+1/n) \Lambda)$ is also a frame  for sufficiently  large
$n$. For every $n$ we now construct a deformation
of the form $\mapn{\lambda} := \alpha_{\lambda,n} \lambda$, where
$\alpha_{\lambda,n}$ is either $1$ or $(1+1/n)$ with roughly half of
the multipliers equal to $1$.  Since only a subset
of $\Lambda $ is moved,  we would think that this deformation is
``smaller'' than the full dilation $\lambda \to (1+\tfrac{1}{n})
\lambda $,   and thus it   should preserve the spanning properties   of the corresponding Gabor
system. Surprisingly, this is completely false. We now indicate how
the coefficients $\alpha _{\lambda ,n}$ need to be chosen. 
Let $\rdd = \bigcup _{l=0}^\infty B_l$ be a partition of $\rdd $ into
the annuli
\begin{align*}
&B_l = \{ z\in \rdd : (1+\tfrac{1}{n})^l \leq |z| <
(1+\tfrac{1}{n})^{l+1} \}, \qquad l \geq 1,
\\
&B_0 := \{ z\in \rdd : |z| < (1+\tfrac{1}{n}) \}
\end{align*}
and define 
$$
\alpha _{\lambda ,n} =
\begin{cases}
  1 & \text{ if } \lambda \in B_{2l}, \\
 1+\tfrac{1}{n} & \text{ if } \lambda \in B_{2l+1}  \, .
\end{cases}
$$
Since $(1+\tfrac{1}{n} ) B_{2l+1} = B_{2l+2}$, the deformed set $\Lambda _n=
\tau _n (\Lambda
) = \{\alpha _{\lambda ,n} \lambda : \lambda \in \Lambda \} $ is contained in $\bigcup
_{l=0}^\infty B_{2l}$ and thus contains arbitrarily large holes. So
$\rho (\Lambda _n) = \infty $ and $D^- (\Lambda _n) = 0$.
Consequently the 
corresponding Gabor system $\cG (g,\Lambda _n)$ cannot be a frame. 
See Figure~\ref{figweird} for a plot of this deformation in dimension $1$.
}
\end{example}
\begin{figure}[!t]
\includegraphics[width=0.6\textwidth,natwidth=442,natheight=252]{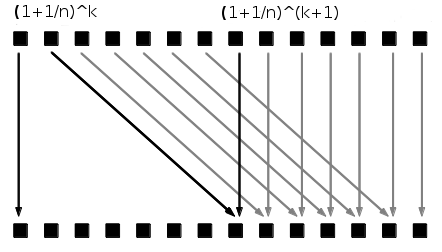}
\caption{A deformation ``dominated'' by the dilation $\lambda \to (1+1/n)\lambda $. }
\label{figweird}
\end{figure}

\section{Appendix}
\label{sec_app}
We finally prove Theorem \ref{th_main_mol}. Both the statement and the  proof
are modeled on 
Sj\"ostrand's treatment of Wiener's  lemma for convolution-dominated
matrices \cite{sj95}.   Several stability results
are built  on his techniques \cite{su07-5, albakr08, shsu09, te10,
  su10-2}.  The following proposition
exploits the flexibility of Sj\"ostrand's methods to 
transfer lower bounds for a matrix  from one value of $p$ to all
others, under the assumption that the entries of 
the matrix decay away from a collection of lines.
\subsection{A variation of Sj\"ostrand's Wiener-type lemma}
Let $G$ be the group $\group := \sett{-1,1}^d$ with coordinatewise
multiplication,  and let $\sigma \in \group$ act on $\Rdst$ by
$\sigma x := (\sigma_1 x_1, \ldots, \sigma_d x_d)$. The group inverse of
$\sigma \in \group$ is $\sigma^{-1}=(-\sigma_1,\ldots,-\sigma_d)$
and the orbit of $x\in \Rdst$ by $\group$
is $\group \cdot x := \sett{\sigma x: \sigma \in \group}$. Consequently,
$\zd = \bigcup _{k \in \bN _0^d } G\cdot k$ is a disjoint union. 
  We note that the cardinality of $G\cdot x$ depends on the number of non-zero
coordinates of $x\in \rd $. 
\begin{prop}
\label{prop_loc}
Let $\Lambda$ and $\Gamma$ be relatively separated subsets of $\Rdst$. Let $A \in \bC^{\Lambda \times \Gamma}$ be a
matrix such that
\begin{align*}
\abs{A_{\lambda,\gamma}} \leq
\sum_{\sigma \in \group} \env(\lambda-\sigma \gamma)
\qquad \lambda \in \Lambda, \gamma \in \Gamma,
\qquad \mbox{for some } \env \in W(L^\infty,L^1)(\Rdst).
\end{align*}
Assume that  there exist a $p\in [1,\infty]$ and  $C_0>0$, such that 
\begin{equation}
  \label{eq:c4}
  \norm{Ac}_{p} \geq C_0 \norm{c}_{p} \quad \quad \text{ for all } c\in \ell
  ^{p} (\Gamma) \, .
\end{equation}
 Then there exists a constant $C >0$ independent of $q$ such that,
 \textrm{for all $q \in [1,  \infty]$}
 \begin{equation}
   \label{eq:c5}
  \norm{Ac}_{q} \geq C \norm{c}_q \quad \quad \text{ for all } c\in \ell
  ^q (\Gamma) \, .
    \end{equation}
In other words, if $A $ is bounded below on some $\ell ^p$, then $A$
is bounded below on  $\ell ^p$ for all $p\in [1,\infty ]$. 
\end{prop}

\begin{proof}
By considering $\sum_{y \in \group\cdot x} \env(y)$ we may assume that
$\env$ is $G$-invariant, i.e. $\env(x)=\env(\sigma x)$ for all $\sigma \in \group$.

\textbf{Step 1.} \emph{Construction of a partition of unity.}  Let $\psi \in
C^\infty(\rd)$ be $G$-invariant and such that $0 \leq \psi \leq 1$, 
$\supp(\psi) \subseteq B_2(0)$ and
\begin{align*}
\sum_{k\in\Zdst} \psi(\cdot-k) \equiv 1.
\end{align*}
For $\varepsilon>0$, define $\psi^\varepsilon_k (x):= \psi(\varepsilon x -k)$,
$I:=\Nst_0^d$, and
\begin{align*}
\varphi^\varepsilon_k := \sum_{j \in \group\cdot k} \psi^\varepsilon_j.
\end{align*}
Since $\zd = \bigcup _{k\in I} G\cdot k$ is a disjoint union,  it
follows  that
$$
\sum_{k \in I} \varphi^\varepsilon_k = \sum _{k\in I} \sum _{j \in
  G\cdot k}  \psi ^\varepsilon_j \equiv 1 \, .
$$
Thus  $\{ \varphi^\varepsilon_k : k\in I\}$ generates a
partition of unity, and it is easy to see that it has the following
additional  properties:
\begin{itemize}
\item $\Phi^\varepsilon := \sum_{k \in I} \left( \varphi^\varepsilon_k \right)^2 \asymp 1$,
\item $0 \leq \varphi^\varepsilon_k \leq 1,$
\item $\abs{\varphi^\varepsilon_k(x) - \varphi^\varepsilon_k(y)} \lesssim \varepsilon \abs{x-y},$
\item $\varphi^\varepsilon_k(x)=\varphi^\varepsilon_k(\sigma x)$ for all $\sigma \in \group$.
\end{itemize}
Combining the last three properties,  we obtain that 
\begin{align}
\label{eq_improved_lip}
\abs{\varphi^\varepsilon_k(x) - \varphi^\varepsilon_k(y)}
\lesssim \min\{1, \varepsilon d(x,\group \cdot y)\},
\end{align}
where $d(x,E) := \inf\sett{\abs{x-y}: y \in E}$.

\textbf{Step 2. } \emph{Commutators.} 
For a matrix $B  =  (B_{\lambda,\gamma})_{\lambda \in \Lambda, \gamma
  \in \Gamma} \in \bC^{\Lambda \times \Gamma}$ we denote the Schur
norm by 
\begin{align*}
\bignorm{B}_{\mbox{Schur}(\Gamma \to \Lambda)}
:= \max\Big\{
\sup_{\gamma \in \Gamma } \sum_{\lambda\in \Lambda } \abs{B_{\lambda,\gamma}},
\sup_{\lambda \in \Lambda } \sum_{\gamma \in \Gamma }  \abs{B_{\lambda,\gamma}}
\Big\}\, .
\end{align*}

Let us assume that $A$ is bounded below on $\ell^p(\Gamma )$. After multiplying
$A$ with a constant, we may  assume that
\begin{align*}
\norm{c}_p \leq \norm{Ac}_p, \qquad c \in \ell^p(\Gamma).
\end{align*}

For given  $\varepsilon>0$ and  $k \in I$  let $\varphi^\varepsilon_k
c := {\varphi^\varepsilon_k}\big| _{\Gamma} \cdot c$ denote the
multiplication operator by $\varphi ^\epsilon _k$ and
$[A,\varphi^\varepsilon_k] = A \varphi ^\epsilon _k -   \varphi
^\epsilon _k A  $ the commutator with $A$. Now 
let us estimate
\begin{align}
\nonumber
\norm{\varphi^\varepsilon_k c}_p & \leq \norm{A \varphi^\varepsilon_k c}_p
\leq \norm{\varphi^\varepsilon_k Ac}_p + \norm{[A,\varphi^\varepsilon_k] c}_p
\\
\nonumber
& \leq
\norm{\varphi^\varepsilon_k Ac}_p +
 \sum_{j\in I} \norm{[A,\varphi^\varepsilon_k]\varphi^\varepsilon_j (\Phi^\varepsilon)^{-1} \varphi^\varepsilon_jc}_p
\\
\label{eq_bound}
& \leq
\norm{\varphi^\varepsilon_k Ac}_p +
K \sum_{j\in I} V^\varepsilon_{j,k} \norm{\varphi^\varepsilon_jc}_p,
\end{align}
where $K = \max _x \Phi ^\epsilon (x)\inv $  and
\begin{align}
\label{eq_def_V}
V^\varepsilon_{j,k} := \norm{[A,\varphi^\varepsilon_k]\varphi^\varepsilon_j}_{\mbox{Schur}(\Gamma \to \Lambda)},
\qquad j,k \in \ I.
\end{align}
The goal of the following steps is to estimate the Schur norm of  the matrix $V^\epsilon $
with entries $V^\epsilon _{j,k}$ and  to show that $\norm{V^\epsilon }
_{\mathrm{Schur}(I\to I)} \to 0$ for $\epsilon \to 0+$.

\Step{Step 3}. \emph{Convergence of the entries $V^\epsilon _{j,k}$.}  We  show that 
\begin{align}
\label{eq_V_unif}
\sup_{j,k \in I} V^\varepsilon_{j,k} \longrightarrow 0,
\mbox{ as }\varepsilon \longrightarrow 0^+.
\end{align}
 We first note that the matrix entries of
$[A,\varphi^\varepsilon_k]\varphi^\varepsilon_j$, for  $j,k \in I$,  are
\begin{align*}
([A,\varphi^\varepsilon_k]\varphi^\varepsilon_j)_{\lambda, \gamma} = 
-A_{\lambda,\gamma} \varphi^\varepsilon_j(\gamma)(\varphi^\varepsilon_k(\lambda)-\varphi^\varepsilon_k(\gamma)),
\qquad \gamma \in \Gamma, \lambda \in \Lambda.
\end{align*}
Using \eqref{eq_improved_lip} and the hypothesis on $A$, we bound the
entries of the commutator by 
\begin{align*}
&\abs{([A,\varphi^\varepsilon_k]\varphi^\varepsilon_j)_{\lambda,\gamma}}
\lesssim
\sum_{\sigma \in \group} \env(\lambda-\sigma \gamma)
\min\{1, \varepsilon d(\lambda,\group\cdot\gamma)\}
\\
&\qquad \leq
\sum_{\sigma \in \group}
\env(\lambda-\sigma \gamma) \min\{1, \varepsilon \abs{\lambda-\sigma\gamma}\}.
\end{align*}
Hence, if we define $\env^\varepsilon(x) := \env(x)
\min\{1,\varepsilon\abs{x}\}$, then by \eqref{eq:c7} 
\begin{align*}
V^\varepsilon_{j,k} = 
\norm{[A,\varphi^\varepsilon_k]\varphi^\varepsilon_j}_{\mbox{Schur}(\Gamma \to \Lambda)}
\lesssim \max\{\rel(\Lambda),\rel(\Gamma)\} \norm{\env^\varepsilon}_{W(L^\infty,L^1)}.
\end{align*}
Since $\env \in W(L^\infty,L^1)$, it follows that
$\norm{\env^\varepsilon}_{W(L^\infty,L^1)} \longrightarrow 0$, as $\varepsilon \longrightarrow 0^+$.
This proves \eqref{eq_V_unif}.

\Step{Step 4}. \emph{Refined estimates for $V^\epsilon _{jk}$.}  For $s \in \Zdst$ let us define
\begin{align}
\label{eq_def_menv}
\menv^\varepsilon(s) := \sum_{t \in \Zdst : \abs{\varepsilon
    t-s}_\infty \leq 5} \sup_{z\in[0,1]^{d}+t} \abs{\env (z) }\, . 
\end{align}
\textbf{Claim:} If $|j-k|>4$ and $\varepsilon \leq 1$, then 
\begin{align*}
V^\varepsilon_{j,k} \lesssim \sum_{s \in \group \cdot j - \group \cdot
  k} \menv^\varepsilon(s)\, .
\end{align*}

If  $\abs{k-j}>4$, then $\varphi^\varepsilon_j(\gamma)
\varphi^\varepsilon_k(\gamma)=0$. Indeed, if this were  not the case, 
then $\varphi^\varepsilon_j(\gamma) \not= 0$ and $\varphi^\varepsilon_k(\gamma) \not=0$. Consequently,
$\abs{\varepsilon \gamma - \sigma j} \leq 2$ and $\abs{\varepsilon
  \gamma - \tau k} \leq 2$ 
for some $\sigma, \tau \in \group$. Hence, $d(k, \group\cdot j) \leq \abs{k-\tau^{-1}\sigma j}=\abs{\tau k-\sigma j}
\leq 4$. Since $k,j \in I$, this implies that $\abs{k-j} \leq 4$,
contradicting the assumption. 

As a consequence, for $\abs{k-j}>4$,   the matrix
entries of $[A,\varphi^\varepsilon_k]\varphi^\varepsilon_j$ simplify
to 
\begin{align*}
([A,\varphi^\varepsilon_k]\varphi^\varepsilon_j)_{\lambda,\gamma} = 
-A_{\lambda,\gamma} \varphi^\varepsilon_j(\gamma)\varphi^\varepsilon_k(\lambda),
\qquad \gamma \in \Gamma, \lambda \in \Lambda \, .
\end{align*}
Hence,  for $\abs{k-j}> 4$ we have the estimate
\begin{align*}
&\abs{([A,\varphi^\varepsilon_k]\varphi^\varepsilon_j)_{\lambda,\gamma}}
\leq 
\sum_{\sigma \in \group} \env(\lambda-\sigma \gamma)
\varphi^\varepsilon_j(\gamma)\varphi^\varepsilon_k(\lambda)
=
\sum_{\sigma \in \group} \env(\lambda-\sigma \gamma)
\varphi^\varepsilon_j(\sigma \gamma)\varphi^\varepsilon_k(\lambda).
\end{align*}

Consequently, for  $\abs{k-j}>4$ we have 
\begin{align*}
&\sup_{\lambda \in \Lambda} \sum_{\gamma \in \Gamma}
\abs{([A,\varphi^\varepsilon_k]\varphi^\varepsilon_j)_{\lambda,\gamma}}
\leq
\sup_{\lambda \in \Lambda} \sum_{\gamma \in \Gamma}
\sum_{\sigma \in \group} \env(\lambda-\sigma \gamma)
\varphi^\varepsilon_j(\sigma \gamma)\varphi^\varepsilon_k(\lambda)
\\
&\quad \lesssim
\sup_{\lambda \in \Lambda} \sum_{\gamma \in \group \cdot \Gamma}
\env(\lambda-\gamma)
\varphi^\varepsilon_j(\gamma)\varphi^\varepsilon_k(\lambda).
\end{align*}
If $\varphi^\varepsilon_j(\gamma) \varphi^\varepsilon_k(\lambda) \not = 0$, then
$\abs{\varepsilon \gamma - \sigma j} \leq 2$ and 
$\abs{\varepsilon \lambda - \tau k} \leq 2$  for some $\sigma, \tau \in \group$.
The triangle inequality implies that
\begin{align}
\label{eq_in}
d(\varepsilon (\lambda - \gamma), \group \cdot j - \group \cdot k) \leq 4.
\end{align}
Hence, we further estimate
\begin{align}
\label{eq_four_terms}
&\sup_{\lambda \in \Lambda} \sum_{\gamma \in \Gamma}
\abs{([A,\varphi^\varepsilon_k]\varphi^\varepsilon_j)_{\lambda,\gamma}}
\lesssim
\sup_{\lambda \in \Lambda} \sum_{s \in \group \cdot j - \group \cdot k}
\sum_{
\stackrel{\gamma \in \group\cdot\Gamma,}{\abs{\varepsilon (\lambda - \gamma)-s}\leq 4}
}
\env(\gamma-\lambda).
\end{align}

For fixed  $s \in \group \cdot j - \group \cdot k$ and $\varepsilon \leq 1$
we bound the inner sum in \eqref{eq_four_terms} by 
\begin{align*}
&\sum_{\gamma \in \group \cdot \Gamma: \abs{\varepsilon (\lambda - \gamma) - s} \leq 4} \env(\gamma-\lambda)
\leq \sum_{t\in\Zdst}
\sum_{
\stackrel{\gamma \in \group \cdot \Gamma: \abs{\varepsilon (\lambda - \gamma) - s} \leq 4}{(\lambda - \gamma) \in
[0,1]^{d}+t}}
\env(\gamma-\lambda)
\\
&\quad \lesssim \rel(\lambda-\group\cdot\Gamma) \sum_{t\in\Zdst :
  \abs{\varepsilon t-s}_\infty \leq 5} \sup_{z\in [0,1]^{d}+t}
\abs{\env (z)}
\\
&\quad \lesssim \rel(\Gamma) \sum_{t\in\Zdst :  \abs{\varepsilon
    t-s}_\infty \leq 5  } \sup_{z\in [0,1]^{d}+t}
\abs{\env (z) }
\\
&\quad\lesssim \menv^\varepsilon(s).
\end{align*}
Substituting  this bound in \eqref{eq_four_terms}, we  obtain
\begin{align*}
\sup_{\lambda \in \Lambda} \sum_{\gamma \in \Gamma}
\abs{([A,\varphi^\varepsilon_k]\varphi^\varepsilon_j)_{\lambda,\gamma}}
\lesssim
\sum_{s \in \group \cdot j - \group \cdot k} \menv^\varepsilon(s).
\end{align*}
Inverting the roles of $\lambda$ and $\gamma$ we obtain a similar
estimate, and the combination yields 
\begin{align}
\label{eq_V_kj}
V^\varepsilon_{j,k} \lesssim \sum_{s \in \group \cdot j - \group \cdot k} \menv^\varepsilon(s),
\quad
\mbox{ for }\abs{j-k}>4 \mbox{ and }\varepsilon \leq 1,
\end{align}
as claimed.

\Step{Step 5}. \emph{Schur norm of $V^\epsilon $.}  Let us show that $\norm{V^\epsilon
}_{\mathrm{Schur}(I\to I)} \to 0$, i.e., 
\begin{align}
\label{eq_schur_jk}
\sup_{k\in I} \sum_{j\in I} V^\varepsilon_{j,k}, \quad
\sup_{j\in I} \sum_{k\in I} V^\varepsilon_{j,k} \longrightarrow 0,
\mbox{ as }\varepsilon \longrightarrow 0^+.
\end{align}
We only treat  the first limit; the second limit is analogous.
Using the definition of  $\menv^\varepsilon$ from \eqref{eq_def_menv}
and the fact $\Theta \in W(L^\infty , L^1)$, we obtain  that
\begin{align}
\label{bound_menv}
\sum_{s\in\Zdst,\abs{s}>6\sqrt{d}} \menv^\varepsilon(s) 
\leq
\sum_{s\in\Zdst,\abs{s}_\infty>6} \menv^\varepsilon(s) 
\lesssim
\sum_{t \in \Zdst, \abs{t}_\infty>1/\varepsilon}
\sup_{z\in[0,1]^{d}+t} \abs{\env (z) }
\longrightarrow 0,
\mbox{ as }\varepsilon \longrightarrow 0^+.
\end{align}
Fix $k \in I$ and use \eqref{eq_V_kj} to estimate
\begin{align*}
&\sum_{j: \abs{j-k} > 6\sqrt{d}} V^\varepsilon_{j,k} \lesssim
\sum_{j: \abs{j-k} > 6\sqrt{d}} \sum_{s \in \group \cdot j - \group \cdot k}
\menv^\varepsilon(s)
\leq
\sum_{\sigma,\tau \in \group}
\sum_{j: \abs{j-k} > 6\sqrt{d}} \menv^\varepsilon(\sigma j - \tau k).
\end{align*}
If $\abs{j-k} > 6\sqrt{d}$ and $j,k\in I$, then also 
$\abs{s}=\abs{\sigma j - \tau k}>6\sqrt{d}$ for all $\sigma , \tau \in
G$. Hence we obtain the bound 
\begin{align*}
\sum_{j\in I: \abs{j-k} > 6\sqrt{d}} V^\varepsilon_{j,k} \lesssim
\sum_{\sigma,\tau \in \group}
\sum_{
\stackrel{s \in \Zdst}{\abs{s}>6\sqrt{d}}} \menv^\varepsilon(s)
\lesssim
\sum_{
\stackrel{s \in \Zdst}{\abs{s}>6\sqrt{d}}} \menv^\varepsilon(s).
\end{align*}
For the sum over $\{j: \abs{j-k} \leq 6 \sqrt{d}\}$ we use the 
bound 
\begin{align*}
&\sum_{j: \abs{j-k} \leq 6\sqrt{d}} V^\varepsilon_{j,k}
\leq
\#\{j: \abs{j-k} \leq 6\sqrt{d}\} \sup_{s,t} V^\varepsilon_{s,t}
\lesssim \sup_{s,t} V^\varepsilon_{s,t}.
\end{align*}
Hence,
\begin{align*}
\sum_{j \in I } V^\varepsilon_{j,k}
\lesssim \sup_{s,t} V^\varepsilon_{s,t} + \sum_{\abs{s}>6\sqrt{d}} \menv^\varepsilon(s),
\end{align*}
which tends to 0 uniformly in  $k$ as $\varepsilon \rightarrow 0^+$
by \eqref{eq_V_unif} and \eqref{bound_menv}. 

\Step{Step 6}. \emph{The stability estimate.} According to the previous step  we may choose $\varepsilon >0$
such that 
\begin{align*}
\norm{V^\epsilon a}_q \leq \frac{1}{2K} \norm{a}_q, \qquad a\in \ell ^q (I)
\end{align*}
uniformly for all $q\in [1,\infty ]$. Using this bound in \eqref{eq_bound},
we obtain that 
\begin{align*}
\left(\sum_{k \in I} \norm{\varphi^\varepsilon_k c}_p^q\right)^{1/q}
\leq
\left(\sum_{k \in I}\norm{\varphi^\varepsilon_kAc}_p^q\right)^{1/q}
+ 1/2 \left(\sum_{k \in I} \norm{\varphi^\varepsilon_k
c}_p^q\right)^{1/q},
\end{align*}
with the usual modification for $q=\infty$. Hence,
\begin{align}
\label{eq_sj_a}
\left(\sum_{k \in I} \norm{\varphi^\varepsilon_k c}_p^q\right)^{1/q}
\leq
2 \left(\sum_{k \in I}\norm{\varphi^\varepsilon_k Ac}_p^q\right)^{1/q}.
\end{align}

\Step{Step 7}. \emph{Comparison of $\ell ^p$-norms.} Let us show that for every $1 \leq q \leq \infty$,
\begin{align}
\label{eq_sj_b}
\left(\sum_{k \in I} \norm{\varphi^\varepsilon_k a}_p^q \right)^{1/q}
\asymp \norm{a}_q,
\qquad a \in \ell^q(\Gamma),
\end{align}
with constants independent of $p,q$, and the usual modification for $q=\infty$.

First note that for fixed $\varepsilon >0$ 
\begin{align*}
N := \sup _{k\in I}\# \supp({\varphi^\varepsilon_k}\big| _{\Gamma}) 
= \sup_{k\in I} \#\set{\gamma \in \Gamma}{\varphi^\varepsilon_k(\gamma) \not= 0} 
< \infty.
\end{align*}
Then for  $q \in [1,\infty]$ we have 
\begin{align*}
\norm{\varphi^\varepsilon_k a}_p
\leq \norm{\varphi^\varepsilon_k a}_1
\leq N \norm{\varphi^\varepsilon_k a}_\infty
\leq N \norm{\varphi^\varepsilon_k a}_q,
\qquad a \in \ell^\infty(\Gamma),
\end{align*}
and similarly
\begin{align*}
\norm{\varphi^\varepsilon_k a}_q
\leq N \norm{\varphi^\varepsilon_k a}_p,
\qquad a \in \ell^\infty(\Gamma).
\end{align*}
As a consequence,
\begin{align}
\label{eq_sj_pq}
\Big( \sum_{k \in I} \norm{\varphi^\varepsilon_k a}_p^q \Big)^{1/q}
\asymp
\Big( \sum_{k \in I} \norm{\varphi^\varepsilon_k a}_q^q \Big)^{1/q},
\qquad a \in \ell^q(\Gamma)
\end{align}
with constants independent of $p$ and $q$ and with  the usual modification for $q=\infty$.

Next 
note that 
\begin{align}
\label{eq_sj_covnum}
\covnum := \sup_{\varepsilon>0} \sup_{x \in \Rdst} \#\set{k \in
  I}{\varphi^\varepsilon_k(x) \not= 0} = \sup_{\varepsilon>0} \sup_{x
  \in \Rdst} \# \{k \in 
  I\cap B_2(\varepsilon x)\}  <\infty.
\end{align}
because  $\supp(\psi) \subseteq B_2(0)$. 
So we  obtain the following simple bound  for all $x \in \Rdst$: 
\begin{align*}
1 = \sum_{k\in I} \varphi^\varepsilon_k(x)
= \sum_{k\in I: \varphi^\varepsilon_k(x) \not= 0} \varphi^\varepsilon_k(x)
\leq \covnum \, \sup_{k \in I} \varphi^\varepsilon_k(x).
\end{align*}
Therefore, for all $x \in \Rdst$, 
\begin{align}
\label{eq_sj_partq}
\frac{1}{\covnum} \leq \sup _{k\in I} \varphi^\varepsilon_k(x)   \leq
\Big( \sum_{k\in I} (\varphi^\varepsilon_k(x))^q \Big)^{1/q} 
\leq \sum_{k\in I} \varphi^\varepsilon_k(x)  = 1 \, .
\end{align}

If $q<\infty$ and  $a \in \ell^q(\Gamma)$, then 
\begin{align*}
\frac{1}{\covnum^q} \sum_{\gamma \in \Gamma} \abs{a_\gamma}^q
\leq \sum_{\gamma \in \Gamma} \sum_{k\in I} (\varphi^\varepsilon_k(\gamma))^q \abs{a_\gamma}^q
\leq \sum_{\gamma \in \Gamma} \abs{a_\gamma}^q,
\end{align*}
which implies that
\begin{align}
\label{eq_sj_qq}
\Big(\sum_{k \in I} \norm{\varphi^\varepsilon_k a}_q^q\Big)^{1/q} =
\Big(\sum_{\gamma \in \Gamma} \sum_{k\in I} (\varphi^\varepsilon_k(\gamma))^q
\abs{a_\gamma}^q \Big)^{1/q} 
\asymp \norm{a}_q,
\end{align}
with constants independent of $q$. The corresponding statement for $q=\infty$ follows similarly.
Finally, the combination of \eqref{eq_sj_pq} and \eqref{eq_sj_qq} yields \eqref{eq_sj_b}.

\Step{Step 8}. We finally combine  the norm
equivalence~\eqref{eq_sj_b} with   \eqref{eq_sj_a} 
and  deduce that for all $1 \leq q \leq \infty$
\begin{align*}
\norm{c}_q \lesssim \norm{Ac}_q,
\end{align*}
with a constant independent of $q$. This completes the proof.
\end{proof}
\begin{rem}
\label{rem_unif_loc}
{\rm  We emphasize that the lower bound guaranteed by Proposition \ref{prop_loc} is
  uniform for all $p$. The constant depends only on the decay
  properties of the  envelope $\Theta$, the lower bound for the given  value of $p$,  and 
on upper bounds for the relative separation of the index sets.}
\end{rem}
\begin{rem}{\rm 
  Note that in the special case $d=1$ and  $\Gamma = \Lambda = \bZ $,
  the assumption of Theorem~\ref{prop_loc} says that
$$
|A_{\lambda, \gamma } | \leq  \Theta (\lambda - \gamma ) + \Theta
(\lambda +\gamma ) \qquad \forall \lambda, \gamma \in \bZ \, ,
$$
i.e., $A$ is dominated by the sum of a Toeplitz matrix and a Hankel
matrix. }
\end{rem}
\subsection{Wilson bases}
A Wilson basis associated with a window $g \in \LtRd$ is an orthonormal basis of $\LtRd$
of the form $\mathcal{W}(g)=\sett{g_\gamma: \gamma=(\gamma_1,\gamma_2) \in \Gamma}$ with
$\Gamma \subseteq \tfrac{1}{2} \Zdst \times \Nst^d_0$ and
\begin{align}
\label{eq_wilson}
g_\gamma = \sum_{\sigma \in \sett{-1,1}^d} \alpha_{\gamma,\sigma} \pi(\gamma_1,\sigma \gamma_2) g,
\end{align}
where $\alpha_{\gamma,\sigma} \in \bC$, $\sup_{\gamma,\sigma} \abs{\alpha_{\gamma,\sigma}} < +\infty$,
and, as before, $\sigma x = (\sigma_1 x_1, \ldots, \sigma_d x_d)$, $x \in
\Rdst$.

There exist Wilson bases associated to functions $g$ in the Schwartz class
\cite{dajajo91} (see also \cite[Chapters 8.5 and 12.3]{gr01}). In this case $\mathcal{W}(g)$
is a $p$-Riesz sequence and a $p$-frame for $M^p(\Rdst)$ for all $p
\in [1,\infty]$. This means that the associated coefficient operator
$C_{\cW}$ defined by $C_{\cW}  =  (\ip{f}{g_\gamma})_\gamma $ is an
isomorphism from $M^p(\rd ) $ onto $ \ell^p(\Gamma)$ for every $p\in
[1,\infty ]$ and that the synthesis operator $C_{\cW}^*c = C_{\cW
}\inv c=  \sum_\gamma
c_\gamma g_\gamma$ is an  isomorphism from $\ell ^p(\Gamma )$ onto
$M^p(\rd )$  for all $p\in [1,\infty ]$. 
(For $p=\infty$ the series converge in the weak* topology).

\subsection{Proof of Theorem \ref{th_main_mol}}
\begin{proof}
Let $\mathcal{W}(g)=\sett{g_\gamma: \gamma \in \Gamma}$ be a Wilson basis with $g \in M^1(\Rdst)$
and $\Gamma \subseteq \tfrac{1}{2} \Zdst \times \Nst^d_0$.

 (a)  We assume that $\{ f_\lambda : \lambda \in \Lambda
\}$ is a set of \tf\ molecules with associated 
coefficient  operator $S$, $Sf := (\ip{f}{f_\lambda} )_{\lambda \in \Lambda}.$

We need to show that if $S$ is bounded below on $M^p(\rd )$  for some
$p \in [1, \infty]$, then it is bounded below for all 
$p \in [1, \infty]$. Since the synthesis operator $C_{\cW }^*$
associated with the Wilson basis  $\mathcal{W}(g)$
is an isomorphism for all $1 \leq p \leq \infty$,
$S$ is bounded below on $M^p(\rd )$, 
if and only if    $SC_{\cW }^* $ is bounded below on $\ell ^p(\Lambda
)$. Thus  it suffices to show that $SC^*_{\cW}$ is bounded below for some
$p \in [1,\infty]$ and then apply  Proposition~\ref{prop_loc}.  The operator $SC^*_{\cW}$  is
represented by the 
matrix $A$ with entries 
\begin{align*}
A_{\lambda, \gamma} := \ip{g_\gamma}{f_\lambda},
\qquad
\lambda \in \Lambda, \gamma \in \Gamma.
\end{align*}
In order to apply Proposition \ref{prop_loc} we provide an adequate envelope.
Let $\Phi$ be the function from \eqref{eq_env_mol},
$\Phi^\vee(z):=\Phi(-z)$ and $\env := \Phi^\vee * \abs{V_g
  g}$. Then    $\env \in
W(L^\infty,L^1)(\Rtdst)$   by Lemma~\ref{lemma_stft}.  
Using \eqref{eq_wilson} and the time-frequency localization of
$\sett{f_\lambda: \lambda \in \Lambda}$ and of $g$ we
estimate
\begin{align*}
&\abs{A_{\lambda, \gamma}} 
= \abs{\ip{f_\lambda}{g_\gamma}} = \abs{\ip{V_g f_\lambda}{V_g g_\gamma}}
\\
&\qquad \leq
\int_{\Rtdst} \abs{\Phi(z-\lambda)} \abs{V_g g_\gamma(z)} dz
\\
&\qquad \lesssim
\sum_{\sigma \in \sett{-1,1}^d}
\int_{\Rtdst} \abs{\Phi(z-\lambda)} \abs{V_g g(z-(\gamma_1,\sigma \gamma_2))} dz
\\
&\qquad=
\sum_{\sigma \in \{-1,1\}^d} \env(\lambda-(\gamma_1,\sigma\gamma_2))
\leq
\sum_{\sigma \in \{-1,1\}^{2d}} \env(\lambda-\sigma \gamma).
\end{align*}
Hence, the desired conclusion follows from Proposition \ref{prop_loc}.

 (b) Here we assume that $\{f_\lambda: \lambda \in \Lambda \}$ is a
set of \tf\ molecules such that the associated synthesis operator 
$S^*c = \sum_{\lambda \in \Lambda} c_\lambda f_\lambda$ is bounded
below on some $\ell ^p(\Lambda )$. 
We must show that $S^*$  is bounded
below for all $p \in [1,\infty]$. Since $C_{\cW} $ is an isomorphism
on $M^p(\rd )$, this is equivalent to saying the operator $C_{\cW}
S^*$ is bounded below on some (hence all) $\ell ^p(\Lambda )$. 

The operator $C_{\cW}S^*$ is represented by the matrix
$B=A^*$ with entries 
\begin{align*}
B_{\gamma,\lambda} := \ip{g_\gamma}{f_\lambda},
\qquad \gamma \in \Gamma, \lambda \in \Lambda \, ,
\end{align*}
and satisfies 
\begin{align*}
\abs{B_{\gamma,\lambda}} \leq \sum_{\sigma \in \{-1,1\}^{2d}} \env(\lambda-\sigma \gamma),
\qquad \gamma \in \Gamma, \lambda \in \Lambda.
\end{align*}
To apply Proposition \ref{prop_loc}, we consider the symmetric envelope
$\env^*(x) = \sum_{y \in \group\cdot x} \env(y), x \in \Rtdst.$ 
Then $\env^* \in W(L^\infty,L^1)(\Rtdst)$,  $\env^*(\sigma x) =
\env(x)$, and 
\begin{align*}
\abs{B_{\gamma,\lambda}} \leq \sum_{\sigma \in \{-1,1\}^{2d}} \env ^*(\lambda-\sigma \gamma)
= \sum_{\sigma \in \{-1,1\}^{2d}} \env ^* (\gamma-\sigma \lambda).
\end{align*}
This shows that we can apply Proposition \ref{prop_loc} and the proof
is complete. 
\end{proof}
\begin{rem}
\label{rem_unif_loc_2}
{\rm As in Remark \ref{rem_unif_loc}, we note that the norm bounds for all $p$ in Theorem \ref{th_main_mol} depend on the
envelope $\Theta$, on upper bounds for $\rel(\Lambda)$ and frame or Riesz basis bounds for a particular value of $p$.}
\end{rem}


\begin{thebibliography}{10}
\bibitem{albakr08}
A.~{A}ldroubi, A.~{B}askakov, and I.~{K}rishtal.
\newblock {S}lanted matrices, {B}anach frames, and sampling.
\newblock {\em {J}. {F}unct. {A}nal.}, 255(7):1667--1691, {O}ctober 2008.

\bibitem{as13}
G.~{A}scensi,
\newblock Sampling Measures for the Gabor Transform.
\newblock {Preprint}, 2013.

\bibitem{asfeka13}
G.~{A}scensi, H.~G. {F}eichtinger, and N.~{K}aiblinger.
\newblock {D}ilation of the {W}eyl symbol and {B}alian-{L}ow theorem.
\newblock {\em {T}rans. {A}mer. {M}ath. {S}oc.},  366(7):3865--3880,
2014. 

\bibitem{bacahela06}
R.~{B}alan, P.~G. {C}asazza, C.~{H}eil, and Z.~{L}andau.
\newblock {D}ensity, overcompleteness, and localization of frames {I}:
{T}heory.
\newblock {\em J. Fourier Anal. Appl.}, 12(2):105--143, 2006. 

\bibitem{bacahela06-1}
R.~{B}alan, P.~G. {C}asazza, C.~{H}eil, and Z.~{L}andau.
\newblock {D}ensity, overcompleteness, and localization of frames. {I}{I}:
{G}abor systems.
\newblock {\em J. Fourier Anal. Appl.}, 12(3):307--344, 2006. 

\bibitem{beurling66}
A.~Beurling.
\newblock Local harmonic analysis with some applications to differential
  operators.
\newblock In {\em Some Recent Advances in the Basic Sciences, Vol. 1 (Proc.
  Annual Sci. Conf., Belfer Grad. School Sci., Yeshiva Univ., New York,
  1962--1964)}, pages 109--125. Belfer Graduate School of Science, Yeshiva
  Univ., New York, 1966.


\bibitem{be89}
A.~{B}eurling.
\newblock {\em {T}he {C}ollected {W}orks of {A}rne {B}eurling. {V}olume 1:
  {C}omplex {A}nalysis. {E}d. by {L}ennart {C}arleson, {P}aul {M}alliavin,
  {J}ohn {N}euberger, {J}ohn {W}ermer.}
\newblock {C}ontemporary {M}athematicians. {B}oston etc.: {B}irkh{\"a}user
  {V}erlag. xx, 475 p./v.1 and and {D}{M} 198.00/set, 1989.

\bibitem{BMO}
J.~Bruna, X.~Massaneda and J.~Ortega-Cerd\`a
\newblock Connections between Signal Processing and Complex Analysis
\newblock {\em Contributions to Science} 2 (3): 345-357, 2003.
  
\bibitem{chdehe99}
O.~{C}hristensen, B.~{D}eng, and C.~{H}eil.
\newblock {D}ensity of {G}abor frames.
\newblock {\em Appl. Comput. Harmon. Anal.}, 7(3):292--304, 1999. 

\bibitem{chr03}
O.~Christensen.
\newblock {\em An introduction to frames and {R}iesz bases}.
\newblock Applied and Numerical Harmonic Analysis. Birkh\"auser Boston Inc.,
  Boston, MA, 2003.

\bibitem{conway90}
J.~B. Conway.
\newblock {\em A course in functional analysis}.
\newblock Springer-Verlag, New York, second edition, 1990.

\bibitem{CGN12}
E.~Cordero, K.~Gr{\"o}chenig and F.~Nicola.
\newblock Approximation of {F}ourier integral operators by {G}abor multipliers.
\newblock {\em J. Fourier Anal. Appl.}, 18(4):661--684, 2012.

\bibitem{CP06}
W.~Czaja and A.~M. Powell.
\newblock Recent developments in the {B}alian-{L}ow theorem.
\newblock In {\em Harmonic analysis and applications}, Appl. Numer. Harmon.
  Anal., pages 79--100. Birkh\"auser Boston, Boston, MA, 2006.


\bibitem{dajajo91}
I.~{D}aubechies, S.~{J}affard, and J.~L. {J}ourn{\'e}.
\newblock {A} simple {W}ilson orthonormal basis with exponential decay.
\newblock {\em SIAM J. Math. Anal.}, 22:554--573, 1991. 

\bibitem{dG13}
M.~A. de~Gosson.
\newblock Hamiltonian deformations of {G}abor frames: {F}irst steps.
\newblock {\em Appl. Comput. Harmon. Anal.}, 38(2):196--221, 2015.

\bibitem{evga92}
L.~C. {E}vans and R.~F. {G}ariepy.
\newblock {\em {M}easure {T}heory and {F}ine {P}roperties of {F}unctions.}
\newblock {S}tudies in {A}dvanced {M}athematics. {C}{R}{C} {P}ress, {B}oca
{R}aton, 1992. 

\bibitem{fe83}
H.~G. {F}eichtinger.
\newblock {B}anach convolution algebras of {W}iener type.
\newblock In {\em {P}roc. {C}onf. on {F}unctions, {S}eries, {O}perators,
  {B}udapest 1980}, volume~35 of {\em {C}olloq. {M}ath. {S}oc. {J}anos
  {B}olyai}, pages 509--524. {N}orth-{H}olland, {A}msterdam, {E}ds. {B}.
  {S}z.-{N}agy and {J}. {S}zabados. edition, 1983.

\bibitem{fe89}
H.~G. {F}eichtinger. 
\newblock Atomic characterizations of modulation spaces through Gabor-type
representations, 
\newblock In ``Proc. Conf. Constructive Function Theory,'' 
{\em Rocky Mountain J. Math.}
19:113--126, 1989.

\bibitem{fegr89}
H.~G. {F}eichtinger and K.~{G}r{\"o}chenig.
\newblock {B}anach spaces related to integrable group representations and their
atomic decompositions, {I}.
\newblock {\em J. Funct. Anal.}, 86(2):307--340, 1989. 

\bibitem{feka04}
H.~G. {F}eichtinger and N.~{K}aiblinger.
\newblock {V}arying the time-frequency lattice of {G}abor frames.
\newblock {\em {T}rans. {A}mer. {M}ath. {S}oc.}, 356(5):2001--2023, 2004.

\bibitem{FS06}
H.~G. Feichtinger and W.~Sun.
\newblock Stability of {G}abor frames with arbitrary sampling points.
\newblock {\em Acta Math. Hungar.}, 113(3):187--212, 2006.

\bibitem{fogr05}
M.~{F}ornasier and K.~{G}r{\"o}chenig.
\newblock {I}ntrinsic localization of frames.
\newblock {\em Constr. Approx.}, 22(3):395--415, 2005. 

 \bibitem{gr91}
K.~{G}r{\"o}chenig.
\newblock {D}escribing functions: atomic decompositions versus frames.
\newblock {\em Monatsh. Math.}, 112(3):1--41, 1991. 

\bibitem{gr01}
K.~{G}r{\"o}chenig.
\newblock {\em {F}oundations of {T}ime-{F}requency {A}nalysis}.
\newblock {A}pplied and {N}umerical {H}armonic {A}nalysis,  {B}irkh{\"a}user {B}oston,
{B}oston, {M}{A}, 2001. 

\bibitem{gr04}
K.~{G}r{\"o}chenig.
\newblock {L}ocalization of frames, {B}anach frames, and the invertibility of
the frame operator.
\newblock {\em J. Fourier Anal. Appl.}, 10(2):105--132, 2004. 

\bibitem{gr07-2}
K.~{G}r{\"o}chenig.
\newblock {G}abor frames without inequalities.
\newblock {\em {I}nt. {M}ath. {R}es. {N}ot. {I}{M}{R}{N}}, (23):{A}rt. {I}{D}
rnm111, 21, 2007. 

\bibitem{heil07}
C.~Heil.
\newblock History and evolution of the density theorem for {G}abor frames.
\newblock {\em J. Fourier Anal. Appl.}, 13(2):113--166, 2007.

\bibitem{MarOrt} 
J. Marzo, J. Ortega-Cerd\`a. 
\newblock Equivalent norms for polynomials on the sphere.  
\newblock {\em Int. Math. Res.Not. IMRN},  no. 5, Art. ID rnm 154, 18p, 
2008. 

\bibitem{OU12}
A.~Olevskii and A.~Ulanovskii.
\newblock On multi-dimensional sampling and interpolation.
\newblock {\em Anal. Math. Phys.}, 2(2):149--170, 2012.

\bibitem{OrtPri}
J.~{O}rtega {C}erd\`a and B.~{P}ridhnani.
\newblock {C}arleson measures and {L}ogvinenko-{S}ereda sets on compact
manifolds.
\newblock {\em {F}orum {M}ath.}, 25(1):151--172, 2013. 

\bibitem{RS95}
J.~Ramanathan and T.~Steger.
\newblock Incompleteness of sparse coherent states.
\newblock {\em Appl. Comput. Harmon. Anal.}, 2(2):148--153, 1995.

\bibitem{Seip92a}
K.~{S}eip.
\newblock {D}ensity theorems for sampling and interpolation in the
{B}argmann-{F}ock space. {I}.
\newblock {\em J. Reine Angew. Math.}, 429:91--106, 1992.

\bibitem{Seip92b}
K.~{S}eip and R.~{W}allst{\'e}n.
\newblock {D}ensity theorems for sampling and interpolation in the
{B}argmann-{F}ock space. {I}{I}.
\newblock {\em J. Reine Angew. Math.}, 429:107--113, 1992. 

\bibitem{Seip}
K.~{S}eip.
\newblock {B}eurling type density theorems in the unit disk.
\newblock {\em {I}nvent. {M}ath.}, 113(1):21--39, 1993. 

\bibitem{shsu09}
C.~{S}hin and Q.~{S}un.
\newblock {S}tability of localized operators.
\newblock {\em {J}. {F}unct. {A}nal.}, 256(8):2417--2439, 2009.

\bibitem{sj95}
J.~{S}j{\"o}strand.
\newblock {W}iener type algebras of pseudodifferential operators.
\newblock {\em {S}{\'e}minaire sur les {\'e}quations aux {D}{\'e}riv{\'e}es
  {P}artielles, 1994-1995, \'Ecole {P}olytech, {P}alaiseau, {E}xp. {N}o.
  {I}{V}}, 21, 1995.

\bibitem{su07-5}
Q.~{S}un.
\newblock {W}iener's lemma for infinite matrices.
\newblock {\em {T}rans. {A}mer. {M}ath. {S}oc.}, 359(7):3099--3123, 2007.

\bibitem{su10-2}
Q.~{S}un.
\newblock {S}tability criterion for convolution-dominated infinite matrices.
\newblock {\em {P}roc. {A}m. {M}ath. {S}oc.}, 138(11):3933--3943, 2010.

\bibitem{te10}
R.~{T}essera.
\newblock {L}eft inverses of matrices with polynomial decay.
\newblock {\em J. Funct. Anal.}, 259(11):2793--2813, 2010. 

\bibitem{young80}
R.~M. Young.
\newblock {\em An introduction to nonharmonic {F}ourier series}.
\newblock Academic Press Inc. [Harcourt Brace Jovanovich Publishers], New York,
  1980.

\end{thebibliography}
\end{document}